\documentclass{mystyle}
\usepackage[utf8]{inputenc}
\usepackage{amsmath}
\usepackage{calc}
\usepackage[dvipsnames]{xcolor}
\usepackage{mathtools}
\usepackage{accents}
\usepackage[all]{xy}   
\usepackage{tikz-cd}
\usepackage{enumitem}
\tikzcdset{arrow style=tikz, diagrams={>=stealth}}                        %
  
\CompileMatrices                            

\UseTips                                    

\input xypic
\usepackage[bookmarks=true]{hyperref}       

\usepackage{amssymb,latexsym,amsmath,amscd}
\usepackage{xspace}
\usepackage{color}
\usepackage{graphicx}
\usepackage{dsfont}
\usepackage{cancel}

\reversemarginpar

\theoremstyle{plain}
\newtheorem{theorem}{Theorem}[section]
\newtheorem*{theorem*}{Theorem}
\newtheorem{proposition}[theorem]{Proposition}

\newtheorem{lemma}[theorem]{Lemma}

\theoremstyle{definition}

\newtheorem{remark}[theorem]{Remark}

\newcommand{\enm}[1]{\ensuremath{#1}}          %
\newcommand{\op}[1]{\operatorname{#1}}
\newcommand{\cal}[1]{\mathcal{#1}}

\renewcommand{\bar}[1]{\overline{#1}}

\newcommand{\NN}{\enm{\mathbb{N}}}

\newcommand{\ZZ}{\enm{\mathbb{Z}}}
\newcommand{\FF}{\enm{\mathbb{F}}}
\renewcommand{\AA}{\enm{\mathbb{A}}}
\newcommand{\PP}{\enm{\mathbb{P}}}

\newcommand{\KK}{\enm{\mathbb{K}}}

\newcommand{\Bb}{\enm{\cal{B}}}

\newcommand{\Dd}{\enm{\cal{D}}}
\newcommand{\Ee}{\enm{\cal{E}}}
\newcommand{\Ff}{\enm{\cal{F}}}
\newcommand{\Gg}{\enm{\cal{G}}}
\newcommand{\Hh}{\enm{\cal{H}}}
\newcommand{\Ii}{\enm{\cal{I}}}

\newcommand{\Ll}{\enm{\cal{L}}}
\newcommand{\Mm}{\enm{\cal{M}}}
\newcommand{\Nn}{\enm{\cal{N}}}
\newcommand{\Oo}{\enm{\cal{O}}}

\newcommand{\Xx}{\enm{\cal{X}}}

\newcommand{\cgb}{Castelnuovo genus bound}

\newcommand{\bpf}{base-point-free }
\renewcommand{\phi}{\varphi}
\renewcommand{\theta}{\vartheta}
\renewcommand{\epsilon}{\varepsilon}

\newcommand{\vnii}{\vskip 2pt \noindent}

\newcommand{\Aut}{\op{Aut}}

\newcommand{\Jac}{\op{Jac}}

\newcommand{\tensor}{\otimes}         

\renewcommand{\to}[1][]{\xrightarrow{\ #1\ }}

\renewcommand{\sl}{\textsl{sl}}

\newcommand{\old}[1]{}
\newcommand{\vni}{\vskip 2pt \noindent}

\newcommand{\w}{\widetilde}
\newcommand{\h}{\ensuremath{\mathcal}}
\newcommand{\HL}{\ensuremath{\mathcal{H}^\mathcal{L}_}}
\newcommand{\HO}{\ensuremath{\mathcal{H}^\mathcal{}_}}

\begin{document}
\title[Hilbert schemes and Hilbert functions of smooth curves in $\mathbb{P}^5$]
{Hilbert scheme and Hilbert functions of smooth curves of degrees at most $15$  in $\mathbb{P}^5$}

\thanks{The first named author is a member of GNSAGA of INdAM (Italy). 
} 
\author[E. Ballico]{Edoardo Ballico}
\address{Dipartimento di Matematica, Universita` degli Studi di Trento\\
Via Sommarive 14, 38123 Povo, Italy}
\email{ballico@science.unitn.it
}

\author[C. Keem]{Changho Keem}
\address{
Department of Mathematics,
Seoul National University\\
Seoul 151-742,  
South Korea}
\email{ckeem1@gmail.com \textrm{and} ckeem@snu.ac.kr}

\thanks{}



\date{\today}
\maketitle
\vspace{-12pt}
\begin{abstract}
Denoting $\h{H}_{d,g,5}$ by the Hilbert scheme of smooth curves of degree $d$ and genus $g$ in $\PP^5$, let $\Hh$  be an irreducible component of $\HO{d,g,5}$.
We  study the 
Hilbert function $h_X:\NN\to \NN$,  $h_X(t):= h^0(\Ii_X(t))$ 
of a general member $X\in\Hh$, especially when the degree of the curve is low;  $d\le 15$.
We also determine the irreducibility of $\HO{d,g,5}$ for $d\le 14$ and study the natural functorial map $\mu :\HO{d,g,5} \to \Mm_g$ in some detail. We describe the fibre $\mu^{-1}\mu(X)$ for a general $X\in\Hh$ as well as determining the projective normality (or being ACM). 
\end{abstract}
\vspace{-12pt}
\section{Introduction}
We denote by $\h{H}_{d,g,5}$ the Hilbert scheme of smooth curves of degree $d$ and genus $g$ in $\PP^5$. Let $\Hh$  an irreducible component of $\HO{d,g,5}$.
In this paper, we  study the 
Hilbert function $h_X:\NN\to \NN$,  $h_X(t):= h^0(\Ii_X(t))$ 
for a general member $X\in\Hh$. We mainly focus our attention to curves of degree $d\le 15$.
We also determine the irreducibility of $\HO{d,g,5}$ for $d\le 14$ and study the moduli map $\mu :\HO{d,g,5} \to \Mm_g$, sending $X\in\HO{d,g,5}$ to its isomorphism class
$\mu (X)=[X]\in\h{M}_g$.
We describe the fibre $\mu^{-1}\mu(X)$ for a general $X\in\Hh$ as well as determining the projective normality (or being ACM). 

Regarding families of curves in $\PP^5$ with prescribed degree and genus $(d,g)$, less is known
compared with curves in lower dimensional projective spaces. 
In this paper we study $\HO{d,g,5}$ when the degree $d$ is rather low, namely $d\le 15$. Besides describing the Hilbert function of a general member in $\Hh$, we determine the irreducibility of $\HO{d,g,5}$ {(for $d\le 14$)}.
 We also study several key properties such as  gonality (and its first scrollar invariant) of a general element as well as characterizing smooth elements of each component. 
 
 Our particular choice of curves of degree $d\le 15$ in $\PP^5$ stems from the fact that -- while  $\HO{d,g,5}$ for $d=15$ (and for every possible genus $g$) has been studied and rather well understood to a certain extent by recent works of the authors \cite{edinburgh}, \cite{bumi24} --  however, for $d\le 14$, not so much attention has been given to $\HO{d,g,5}$ except for a couple of  exceptional cases, e.g. when $g$ is low with respect to $d$. 
Needless to say, as far as the author knows, only scattered results for the degree $d\ge16$ have been known, e.g.  in \cite{landscape}.

The organization of this paper is as follows. In the next section, we prepare some preliminaries relevant to our study. We also quote or prove several lemmas which we will be using later, such as smoothing nodal reducible curves in $\PP^5$ adopted for our specific situation. In section 3, we study the Hilbert function of curves of degree $d=15$, which have not been treated in 
\cite{bumi24} or \cite{edinburgh}. In subsequent sections, we study $\HO{d,g,5}, d\le 14$ in some detail, 
including the irreducibility and the description of a general curve.

\subsection{Notation and conventions}
For notation and conventions, we  follow those in \cite{ACGH} and \cite{ACGH2}; e.g. $\pi (d,r)$ is the maximal possible arithmetic genus of an irreducible,  non-degenerate and reduced curve of degree $d$ in $\PP^r$ which is usually referred as the first Castelnuovo genus bound. We shall refer to  non-degenerate, irreducible curves $X\subset \PP^r$ with the maximal possible genus $g=\pi (d,r)$ as {\it extremal curves}. $\pi_1(d,r)$ is the so-called the second Castelnouvo genus bound which is the maximal possible arithmetic genus of  an irreducible, non-degenerate and reduced curve of degree $d$ in $\PP^r$ not lying on a  surface of minimal degree $r-1$; cf. \cite[p. 99]{he}, \cite[p. 123]{ACGH}. We shall call curves $X\subset\PP^r$ of degree $d$ and (arithmetic) genus $g$ such that $\pi_1(d,r)<g<\pi(d,r)$ {\it nearly extremal curves}.

Following classical terminology, a linear series of degree $d$ and dimension $r$ on a smooth curve $C$ is denoted by $g^r_d$.
A base-point-free linear series $g^r_d$ ($r\ge 2$) on a smooth curve $C$ is called {\it birationally very ample} when the morphism 
$C \rightarrow \mathbb{P}^r$ induced by  the $g^r_d$ is generically one-to-one onto (or is birational to) its image curve.  A base-point-free linear series $g^r_d$ on $C$  is said to be compounded of an involution ({\it compounded} for short) if the morphism induced $g^r_d$ gives rise to a non-trivial covering map $C\rightarrow C'$ of degree $k\ge 2$. 

We  also recall the following standard set up and notation; cf. \cite{ACGH2}  or \cite[\S 1 and \S 2]{AC2}.
Let $\mathcal{M}_g$ be the coarse moduli space of smooth curves of genus $g$. Given an isomorphism class $[C] \in \mathcal{M}_g$ corresponding to a smooth irreducible curve $C$, there exist a neighborhood $U\subset \mathcal{M}_g$ of  $[C]$ and a smooth connected variety $\mathcal{M}$ which is a finite ramified covering $h:\mathcal{M} \to U$, as well as  varieties $\mathcal{C}$ and $\mathcal{G}^r_d$ proper over $\mathcal{M}$ with the following properties:
\begin{enumerate}
\item[(1)] $\xi:\mathcal{C}\to\mathcal{M}$ is a universal curve, i.e. for every $p\in \mathcal{M}$, $\xi^{-1}(p)$ is a smooth curve of genus $g$ whose isomorphism class is $h(p)$,
\item[(2)] $\mathcal{G}^r_d$ parametrizes the couples $(p, \mathcal{D})$, where $\mathcal{D}$ is possibly an incomplete linear series of degree $d$ and dimension $r$ on $\xi^{-1}(p)$.
\item[(3)] $\mathcal{W}^r_d$ parametrizes the couples $(p, \mathcal{D})$, where $\mathcal{D}$ is a complete linear series of degree $d$ and dimension at least $r$ on $\xi^{-1}(p)$.
\end{enumerate}
We set 
\begin{equation}
\lambda(d,g,r):= 3g-3+\rho (d,g,r),
\end{equation}
the minimal possible dimension of any component of $\Gg^r_d$ and 
\begin{equation}\label{md}
\h{X}(d,g,r):=3g-3+\rho (d,g,r)+\dim\Aut(\PP^r),
\end{equation}
which is the minimal possible dimension of any component of $\HO{d,g,r}$. 

Throughout we work over an algebraically closed field $\KK$ of characteristic zero.
\vspace{-12pt}
\section{Preliminaries and  relevant results}
\subsection{Curves on rational surfaces}

We often need to deal with curves on a quartic surface  in 
$S\subset\PP^5$. The following description of surfaces of minimal degree is well known, which we quote and remind readers just
for fixing notation.

\begin{remark}\label{cone1}
A quartic surface $S\subset\PP^5$  is one of the following types.
\begin{itemize}
\item[(0)] $S$ is a Veronese surface. For any smooth curve $X\subset S$, $\deg X=2s$ and $g=p_a(X)=\binom{s-1}{2}$.
\item[(1a)] $S\subset\PP^5$ is a cone over a 
rational normal curve $R\subset H\cong\PP^{4}$. Let $P\notin H$ be the vertex of the cone $S$.
Recall that $S$ is the image of morphism $\FF_{4}=\mathbb{P}(\h{O}_{\PP^1}\oplus\h{O}_{\PP^1}(4)) \xrightarrow{|C_0+4f|} S\subset\PP^5$  contracting $C_0$ to the vertex $P$; $C_0^2=-4, C_0\cdot f=1$, $f^2=0$ where $f$ is the fibre of $\FF_{4}\rightarrow\PP^1$. 
Let $C\subset S$ be a reduced and non-degenerate curve of degree $d$ and and let $\w{C}$ be the strict transformation of $C$ under  $\widetilde{S}\cong\mathbb{F}_{4}\to S$. 
~Setting $k=\w{C}\cdot f$ we have $\w{C}\equiv kC_0+df$ 
and
\begin{equation}\label{conevertex1}
0\le \w{C}\cdot C_0=(kC_0+df)\cdot C_0=d-4k=m
\end{equation} where $m$ is the multiplicity of $C$ at the vertex  $P$ of $S$.
By adjunction,
\begin{equation}\label{cone}
p_a(C)=1/2\, \left( k-1 \right)  \left( 2\,d-4k-2 \right).
\end{equation}
\item[(1b)]
For $11\le d\le 15$, there are only two possibilities  satisfying \eqref{conevertex1} and \eqref{cone} with $m\in\{0,1\}$;
$(d,k,m,g)\in\{(12,3,0,10), (13,3,1,12)\}$.  Thus a smooth curve in $\PP^5$ of degree $d$ with $11\le d\le 15$ lying on a rational normal cone is {\it trigonal} having genus $g=10$ or $g=12$.
\item[(2a)] $S$ is a smooth surface scroll, $\mathrm{Pic}(S)=\mathbb{Z}H\oplus\mathbb{Z}{L}$
where $H$ (resp. $L$) is the class of a hyperplane section (resp. the class a line of the ruling).
Let $X\subset S$ be a curve and assume$X\in|aH+bL|$:
\begin{equation}\label{sd}
\deg X=(aH+bL)\cdot H=4a+b
\end{equation}
\begin{equation}\label{sg}
p_a(X)=2{(a-1)(a-2)}+(3+b)(a-1)
\end{equation}
\begin{equation}\label{sdn}
 \dim|aH+bL|=2a \left(a +1\right)+\left(a +1\right) \left(b +1\right)-1
 \end{equation}
 \item[(2b)] 
 In this case, we have another set up such as
 either 
 $$S\cong\FF_0\cong\PP(\Oo(2)\oplus\Oo(2))\cong\PP^1\times\PP^1=Q\lhook\joinrel\xrightarrow{|\Oo_Q(1,2)|}\PP^5$$
 or 
 $$S\cong\FF_2\cong\PP(\Oo(3)\oplus\Oo(1))\lhook\joinrel\xrightarrow{|\Oo_{\FF_2}(h+3f)|}\PP^5.$$  
 If $X\in|\Oo_Q(a,b)|$, 
 \begin{equation}\label{F0} d=\deg X=2a+b, \, g=(a-1)(b-1).
 \end{equation}
 If $X\in|\Oo_{\FF_2}(ah+bf)|$, 
 \begin{equation}\label{F2}d=\deg X=a+b, \,g=(a-1)(b-a-1).
 \end{equation}
\item[(2c)] For $(d,g)\in\{(15,15),(15,14), (13,11), (13,10)\}$, there is no integer pair $(a,b)\in\NN\times\NN$ satisfying \eqref{F0} or \eqref{F2}. 
\end{itemize}
\end{remark}
\begin{remark} \label{bih}Let $X\stackrel{\eta}{\rightarrow} E$ be a double covering of a curve $E$ of genus $h\ge 1$. Let $\h{E}=g^s_{e}$ be a {\it non-special} linear series on $E$. Assume that $|\eta^*(g^s_{e})|=g^s_{2e}$. Then the base-point-free part of the complete  
$|K_X(-\eta^*(g^s_{e}))|$ is compounded; the proof is straightforward which we omit.
\end{remark}
The following inequality - known as Castelnuovo-Severi inequality - shall be used occasionally even without explicit mention; cf. \cite[Theorem 3.5]{Accola1}.
\begin{remark}[Castelnuovo-Severi inequality]\label{CS} Let $C$ be a curve of genus $g$ which admits coverings onto curves $E_h$ and $E_q$ with respective genera $h$ and $q$ of degrees $m$ and $n$ such that these two coverings admit no common non-trivial factorization; if $m$ and $n$ are primes this will always be the case. Then
$$g\le mh+nq+(m-1)(n-1).$$ 
\end{remark}

\subsection{Some remarks on scrollar invariats}
When we deal with a curve $X$ on Hirzebruch surfaces $\FF_0, \FF_2\subset\PP^5$ embedded by appropriate linear systems as in Remark \ref{cone1}\,(d),  we utilize the following fact which enables us to deduce the uniqueness of a very ample and complete $|\Oo_X(1)|$; cf. \cite[Proposition 2.3]{edinburgh}.
\begin{proposition}\label{z5}
Let $Q\subset\PP^3$ be a smooth quadric surface. 
Fix integers $3\le a<b\le 2a-1$ and  a smooth $X\in |\Oo_Q(a,b)|$. Then $X$ is $a$-gonal with a unique $g^1_a$, no base-point-free $g^1_c$ for $a<c<b$ and a unique base-point-free $g^1_b$, the pencil induced by $|\Oo_Q(1,0)|$.
\end{proposition}
\begin{remark}  Let $X$ be a smooth curve with a base-point-free pencil $g^1_a, a\ge 3$.
The first (or the last if one prefers) scrollar invariant of $X$ with respect to $g^1_a$
is the integer 
$$m(X,g^1_a):=\max\{k| \dim|kg^1_a|=k\}-1.$$
When $a=3$ and $g\ge5$, $m(X,g^1_a)$ is classically known as the Maroni invariant of a trigonal curve.
When $g^1_a$ is the unique pencil computing the gonality of $X$, we use the notation $m(X)$ instead of $m(X,g^1_a)$.
\end{remark}
\begin{lemma}\label{maroni}Let $X\subset \PP^5$ be a smooth curve of degree $d$ and genus $g$ lying on a smooth surface scroll.

\begin{itemize} 
\item[\rm{(i)}] Assume $X\in|\Oo_Q(a,b)|$, $3\le a\le b$. Let $g^1_a=|\Oo_Q(0,1)|_{|X}$. We have 
$\,m(X,g^1_a)=b-2$.
\item[\rm{(ii)}]
Assume $X\in|\Oo_{\FF_2}(ah+bf)|$ and $6\le 2a\le b$. Let 
$g^1_a=|\Oo_{{\FF}_2}(f)|_{|X}$
be the pencil cut out by the fiber $\FF_2\to\PP^1$. We have $\,m(X,g^1_a)=b-2a$.
\end{itemize}
\end{lemma}
\begin{proof}
\begin{enumerate}
\item[(i)] 
From the sequence
\begin{equation}\label{eqa11}
0 \to \Oo_Q(-a,-b+t)\to \Oo_Q(0,t) \to \Oo_X(0,t)\to 0,
\end{equation}
$$h^0(Q,\Oo_Q(-a,-b+t)) =0, h^1(Q,\Oo_Q(0,t)) =0, h^0(Q,\Oo_Q(0,t)) =t+1 ~\,\forall t\ge -1.$$ For all $ t\in \NN$, the long cohomology exact sequence of \eqref{eqa11} gives $\dim |tg^1_a| =t$ if and only $h^1(Q,\Oo_Q(-a,-b+t)) =0$. By the  K\"{u}nneth formula, 
$$h^1(Q,\Oo_Q(-a,-b+t))=h^1(\PP^1,\Oo_{\PP^1}(-a))\cdot h^0(\Oo_{\PP^1}(-b+t)).$$
Thus $h^1(Q,\Oo_Q(-a,-b+t)) =0$ if and only if $t\le b-1$, hence $m(X,g^1_a)=b-2$.
\item[(ii)] 
For each $t\ge 1$, $$h^0(\FF_2,\Oo_{\FF_2}(-ah+(-b+t)f))=0, h^0(\FF_2,\Oo_{\FF_2}(tf)) =t+1, h^1(\FF_2,\Oo_{\FF_2}(tf)) =0.$$
From the long cohomology exact sequence of the sequence
\begin{equation*}\label{eqa22}
0\to \Oo_{\FF_2}(-ah+(-b+t)f)\to \Oo_{\FF_2}(tf) \to \Oo_X(tf)\to 0
\end{equation*}
$h^0(X,\Oo_X(tf)) =t+1$ if and only if $h^1(\FF_2,\Oo_{\FF_2}(-ah+(-b+t)f)) =0.$ 
From duality, we have $$h^1(\FF_2,\Oo_{\FF_2}(-ah+(-b+t)f))=h^1(\FF_2,((a-2)h+(b-4-t)f)).$$
Thus $h^1(\FF_2,\Oo_{\FF_2}(-ah+(-b+t)f))=0$ if and only if $t\le b-2a+1$; if $x\ge 0$, recall $h^1(\Oo_{\FF_2}(xh+yf)) =0$ if and only if $y\ge 2x-1$ by \cite[Lemma 2.1]{Antonelli}).
 Therefore we have $m(X,g^1_a)=b-2a.$
\end{enumerate}
\vspace*{-\baselineskip}
\end{proof}
\subsection{Some generalities on Hilbert functions}
In this paper, we determine the Hilbert function of smooth curves of degree $d\le 15$,  which is one of our main 
objects of study.  We collect
several  generalities and prove necessary lemmas.
\begin{remark}\label{lars}
Take an irreducible component $\Hh\subset \HO{d,g,5}$. Assume  that the moduli map $\Hh\stackrel{\mu}{\to}\Mm_g$ is dominant and hence $\rho(d,g,5)\ge 0$ by Brill-Noether theorem; cf. \cite{he} \cite[Theorem 19.9]{EH}. Take a general $X\in \Hh$. Since $X$ has general moduli, $h^1(X,\Oo_X(2)) =0$; cf. \cite{gieseker}. From $\rho(d,g,r)\ge 0$, $\deg\Oo_X(t)\ge 2g-1$ for $t\ge 3$ hence 
$h^0(X,\Oo_X(t)) =td+1-g$ for all $t\ge 2$. By \cite{l8},  the restriction map $\rho_{X,t}: H^0(\Oo_{\PP^5}(t))\to H^0(X,\Oo_X(t))$ has maximal rank, determining the Hilbert function of $X$, i.e. $h^0(\PP^5,\Ii_X(1))=0$ and $\forall t\ge 2$
\begin{align}
h^0(\PP^5,\Ii_X(t)) &= \max\{0,\tbinom{t+5}{5} -1-td+g\}\label{hil}\\h^1(\PP^5,\Ii_X(t)) &=\max\{0,dt+1-g -\tbinom{t+5}{5}\}.\nonumber
\end{align}
\end{remark}
\begin{remark}\label{basic1}
(i) For an integral and non-degenerate curve $X\subset \PP^n$ of degree $d$ and arithmetic genus $\pi(d,n)$, $X$ is ACM and is contained in a surface of minimal degree $(n-1)$ which is either a smooth surface scroll, a cone over a rational normal curve of $\PP^{n-1}$ or a Veronese surface if $n=5$; \cite[Theorem 3.7]{he}. For $t\in\NN$,  $h^1(\Ii_X(t)) =0$ and  
\begin{equation}\label{exthil}h^0(\PP^n,\Ii_X(t)) =\binom{n+t}{n} -td-1+\pi(d,n)
-h^1(X,\Oo_X(t)).
\end{equation}
\end{remark}
\begin{remark}\label{pre2}
(i) Let $Y\subset \PP^{r}$, $r\ge 2$ be an integral and non-degenerate curve with $\deg (Y)=d$. Fix a general hyperplane $H\subset \PP^r$ and set $A:= Y\cap H$.
The scheme $A$ is a set of $d$ points spanning $H$. As in \cite[Ch. III]{he} for each $t\in\NN$, set $h_A(t):= d-h^1(H,\Ii_{A,H}(t))$. Since $A$ spans $H$, $h_A(1) =r$.
Assume $h_A(t) <d$, i.e. assume $h^1(H,\Ii_{A,H}(t)) >0$. By \cite[Corollary 3.5]{he} we have $$h_A(t+1)\ge \min\{d,h_A(t)+h_A(1)-1\} =\min\{d,h_A(t)+r-1\},$$ i.e.
$h^1(H,\Ii_{A,H}(t+1)) \le \max\{0,h^1(H,\Ii_{A,H}(t))+1-r\}$. Thus, the sequence $$\{h^1(H,\Ii_{Y\cap H,H}(t))\}_{t\ge 0}$$ is  non-increasing and {\it strictly decreasing} until zero.
\vni
(ii)
Now take $r=5$ and $X\in \Hh_{d,g,5}$. Since $X\cap H$ is formed by $d$ points spanning $H$, $h^1(H,\Ii_{X\cap H,H}(t)) =0$ for all $d\le 4t+1$. Thus,
 $h^1(H,\Ii_{X\cap H,H}(t)) =0$ for all 
 
 (1) $t\ge 2$ if $d\le 9$
 
 (2) $t\ge 3$ if $10\le {d}\le 13$
 
 (3) $t\ge 4$ if $14\le d\le 17$.
 \vnii
 (iii) From the long cohomology  sequence of the exact sequence
\begin{equation}\label{eqppp2}
0\to \Ii_{X}(t-1)\to \Ii_X(t)\to \Ii_{X\cap H,H}(t)\to0,
\end{equation} we have 
$h^1(\Ii_X(t)) =0$ if $h^1(\Ii_{X\cap H,H}(t)) =h^1(\Ii_{X}(t-1)) =0$.
Hence by (ii) above,  for $d\le 9$ (resp. $10\le d\le 13$, resp. $14\le d\le 17$)  $X$ is ACM if and only if it is linearly normal (resp. linearly normal and $h^1(\Ii_X(2)) =0$, resp. $h^1(\Ii_X(t)) =0$ for $t=1,2,3$).
\vnii
(iv) For an integral non-degenerate curve $X\subset\PP^n$, $H$ a general hyperplane, let $u_{X,t}: H^0(\PP^5,\Ii_X(t)) \to H^0(H,\Ii_{X\cap H,H}(t))$ denote the restriction map. 
Since $h^0(\Ii_X(1)) =0$, the long cohomology exact sequence of \eqref{eqppp2} for $t=2$ gives the injectivity of $u_{X,2}$. If $X$ is linearly normal,  i.e. if $h^1(\Ii_X(1))=0$, the same long cohomology exact sequence gives the surjectivity of $u_{X,2}$. 
\end{remark}
\begin{remark}\label{aa1}
Let $Z\subset H =\PP^{n-1}$ be a finite set with  $d:= \#Z$ having the uniform position property and spanning $H$.  
\vnii
(i) Assume
$d\ge 2n+1 ~\mathrm{ and }~h^0(H,\Ii_{Z,H}(2))) \ge \binom{n+1}{2} -2(n-1)-1.$
Then $Z$ is contained in a rational normal curve in $H$; \cite[Lem. 3.9, Prop. 3.19]{he}. 
\noindent
\vni
(ii) Assume further $d\ge 2n+3 \mathrm{~ and ~} h^0(H,\Ii_{Z,H}(2)) \ge \binom{n+1}{2} -2n.$
Then $Z$ is contained in either an irreducible curve with arithmetic genus $1$ of degree $n$ (elliptic normal curve) or a rational normal curve in $H$; \cite[Prop. 3.19 and 3.20]{he}. 
\end{remark}
\begin{lemma}\label{oo1}
Fix an integer $n\ge 5$ and let $Y\subset \PP^{n+1}$ be an extremal curve contained in a minimal degree smooth surface $S_1$. Take any $o\in \PP^{r+1}\setminus \mbox{Sec}(S_1)$ such that the linear projection $\ell_o: \PP^{n+1}\setminus \{o\}\to \PP^n$ induces an embedding of $S_1$ into $\PP^n$. Set $X:= \ell_o(Y)$. Then $h^1(\Ii_X(t)) =0$ for all $t\ge 2$ and $h^1(\Ii_X(1)) =1$.
\end{lemma}
\begin{proof}
Since $Y$ is linearly normal, $h^1(\Ii_X(1))=1$. 
Set $S:= \ell_o(S_1)$.
For any $t\ge 2$,  we consider\vni
\begin{tikzcd}
		H^0(\PP^{n+1},\Oo_{\PP^{n+1}}(t)) \arrow[r, "\rho_{S_1,t}"] 
		& H^0(S_1,\Oo_{S_1}(t)) \arrow[r, "\rho_{S_1,Y,t}"] \arrow[d, "\cong"'] & H^0(Y,\Oo_Y(t)) \arrow[d, "\cong"] \\
		H^0(\PP^n,\Oo_{\PP^n}(t)) \arrow[r, "\rho_{S,t}"]                 & H^0(S,\Oo_S(t)) \arrow[r, "\rho_{S,X,t}"]                 & H^0(X,\Oo_X(t))               
	\end{tikzcd}
\noindent	
\vni
with obvious horizontal restriction maps.  By assumption, $\deg S=n$ and $S$ is smooth isomorphic to $S_1$ as an abstract surface. By a theorem of R. Lazarsfeld \cite[Theorem, p. 423]{laz}, 
$h^1(\Ii_S(t)) =0$ for all $t\ge 2$.  Hence $\rho_{S,t}$ is surjective for all $t\ge 2$.

Fix an integer $t\ge 2$ and let $\rho_{X,t}: H^0(\Oo_{\PP^n}(t))\to H^0(X,\Oo_X(t))$ be the restriction map. Since $\rho_{X,t}=\rho_{S,X,t}\circ \rho_{S,t}$
and $\rho_{S,t}$ is surjective, $h^1(\Ii_X(t)) =0$ if and only if $\rho_{S,X,t}$ is surjective.
Since $\ell _o$ is linear projection mapping $S_1$ isomorphically onto $S$ and $Y$ isomorphically onto $X$,  $\ell _o$ induces isomorphisms $H^0(S_1,\Oo_{S_1}(t))\cong H^0(S,\Oo_S(t))$ and $H^0(Y,\Oo_Y(t))\cong H^0(X,\Oo_X(t))$ commuting with the restriction maps. Hence $\rho_{S,X,t}$ is surjective if and only $\rho_{S_1,Y,t}$ is surjective. Since $Y$ is ACM (\cite[Th. 3.7]{he}), $\rho_{S_1,Y,t}$ is surjective and so is $\rho_{S,X,t}$ hence  the conclusion.
\end{proof}

\begin{lemma}\label{o1oo1}
Fix $d\in \{14,15\}$. Let $Y\subset \PP^5$ be an integral and non-degenerate curve of degree $d$. Let $H\subset \PP^5$ be a general hyperplane.
If $h^1(H,\Ii_{Y\cap H,H}(3)) >0$, then $Y\cap H$ is contained in an integral curve which is either a rational normal curve or a linearly normal curve of arithmetic genus $1$, the latter being possible only for $d=15$.
\end{lemma}
\begin{proof}
By the uniform position lemma \cite[pp. 111--112]{ACGH}, the monodromy group of the generic hyperplane section of $Y$ is the full symmetric group.

(a) In this step, we assume $h^0(H,\Ii_{Y\cap H,H}(2))\le 4$. 

\noindent
Since $d\ge h^0(H,\Oo_H(2))-4$, there is $S\subset Y\cap H$ such that $\#S=d-4$ and $h^0(H,\Ii_{S,H}(2)) = 4$. Hence for any $A\subset S$ with $\#A =d-5$, the base locus of $|\Ii_A(2)|$ does not contain the remaining point of $S\setminus A$. Since the  monodromy group of the generic hyperplane section $Y\cap H$ is $S_d$, we have $h^1(H,\Ii_{A,H}(2)) =0$ and the base locus of $|\Ii_{Y\cap H,H}(2)|$ contains no other point of $Y\cap H$.
We order the points $p_1$, $p_2$, $p_3$, $p_4$, $p_5$ of $Y\cap H \setminus A$. Set $A_0:= A$. For $i=1,\dots ,5$ set $A_i:= A\cup \{p_1,\dots ,p_i\}$.
Since $h^1(H,\Ii_{A,H}(2)) =0$, we have $h^1(H,\Ii_{A,H}(3)) =0$. 
Under our assumption $h^0(H,\Ii_{Y\cap H,H}(2))\le 4$, 
we claim that  $h^1(H,\Ii_{A_i,H}(3))=0$, $i=1,\dots ,5$. 
Take $Q\in |\Ii_A(2)|$ such that $Q\cap Y\cap H =A$.
Since $Y\cap H$ is in linearly general position, { there is a hyperplane $M\cong\PP^3\subset H$ such that $\{p_1,\cdots,p_{i-1}\}\subset M$ and $Y\cap H\cap M = (Y\cap H)\setminus A_{i}$. The cubic hypersurface $Q\cup M$ contains $A_{i-1}$, but not $A_i$.}

(b) By step (a), we may assume $h^0(H,\Ii_{Y\cap H,H}(2)) \ge 5$.   Call $\Bb$ the base locus of $|\Ii_Y(2)|$. By Remark \ref{aa1}(ii),  $Y\cap H$ is contained in an irreducible curve $C$, which is either a rational normal curve or a linearly normal curve of arithmetic genus $1$. Assume that $C$ is a linearly normal curve of arithmetic genus $1$. Since $C$ is arithmetically Cohen-Macaulay, 
$h^1(H,\Ii_{C,H}(3)) =0$ and the restriction map 
$\rho_C: H^0(H,\Oo_H(3)) \to H^0(C,\Oo_{C}(3))$ 
is surjective. We have $\rho_{Y\cap H} =\rho_{C,Y\cap H}\circ \rho_C$,

$H^0(H,\Oo_H(3))\stackrel{\rho_{Y\cap H}}{\to} H^0(Y\cap H,\Oo_{Y\cap H}(3))\to H^1(H,\Ii_{Y\cap H,H}(3))$

$H^0(C,\Oo_C(3))\stackrel{\rho_{C,Y\cap H}}{\to} H^0(Y\cap H,\Oo_{Y\cap H}(3))\to H^1(C,\Ii_{Y\cap H,C}(3))$
\vni
where $\rho_{Y\cap H}$, $\rho_{C,Y\cap H}$ are obvious restriction maps. By the surjectivity of $\rho_C$,
we get  $$h^1(H,\Ii_{Y\cap H,H}(3))=h^1(C,\Ii_{Y\cap H,C}(3)).$$

Since $C$ is an integral curve of arithmetic genus $1$ and $\deg(\Oo_C(3))=15$, {$h^1(C,\Ii_{Z,C}(3))=0$ for every zero-dimensional scheme $Z\subset C$ of degree at most $14$, by Riemann-Roch.}
\end{proof}
We recall some standard notation concerning linear systems and divisors on a blown up projective plane. Let $\PP^2_s$ be the rational surface $\PP^2$ blown up at $s$ general points. Let $e_i$ be the class of the exceptional divisor
$E_i$ and $l$ be the class of a line $L$ in $\PP^2$. For integers  $b_1\ge b_2\ge\cdots\ge b_s$, let $(a;b_1,\cdots, b_i, \cdots,b_s)$ denote class of the linear system $|aL-\sum b_i E_i|$ on $\PP^2_s$.  By abusing notation we use the expression $(a;b_1,\cdots, b_i, \cdots,b_s)$ for the divisor $aL-\sum b_i E_i$ and $|(a;b_1,\cdots, b_i, \cdots,b_s)|$ for the linear system $|aL-\sum b_i E_i|$. We use the convention 
\vspace{-3pt}
$$(a;b_1^{s_1},\cdots,b_j^{s_j},\cdots,b_t^{s_t}), ~ \sum s_j=s$$ 
when  $b_j$ appears $s_j$ times consecutively  in the linear system $|aL-\sum b_i E_i|$.

\begin{remark}\label{delpezzo1} (i) Let $T\cong\PP^2_4\subset\PP^5$ be a smooth del Pezzo surface. Let 
$L:=(a;b_1,\cdots,b_4)$ be a very ample line bundle and $X\in|L|$. We have 
\begin{equation}\label{delPC}\deg X =3a-\sum b_i, X^2=a^2-\sum b_i^2=2g-2-K_S\cdot X=2g-2+\deg X.
\end{equation}
The dimension of such family of curves lying on smooth del Pezzo surfaces can be counted as follows regardless of particular choice of classes $(a;b_1,\cdots,b_4)$. This will enable us to compute the Hilbert function of a {\it general member} in a component of $\HO{d,g,5}$.
By Riemann-Roch on $T$ we have
\begin{eqnarray*}\label{delpezzo}
h^0(T,L)&=&h^1(T,L)-h^2(T,L)+\chi(T)+\frac{1}{2}(L^2-L\cdot\omega_T)\\
&=&h^1(T,L)-h^2(T,L)+1+\frac{1}{2}(X^2+\deg X)\\
&=&h^1(T,L)-h^2(T,L)+1+(g-1+\deg X).
\end{eqnarray*}
\vni
By Serre duality, 
$$h^1(T,L)=h^1(T, \omega_T\otimes L^{-1})=h^1(T,\mathcal{O}_T(-(a+3)l+\sum(b_i+1)e_i).$$
Since $L$ is ample, $E:=(a+3)l-\sum(b_i+1)e_i$ is (very) ample \cite[Theorem]{sandra}, thus $$h^1(T,L)=h^1(T,\mathcal{O}_T(-(a+3)l+\sum(b_i+1)e_i))=0$$ by Kodaira's vanishing theorem. 
We  note the divisor $-E$ is not linearly equivalent to an effective divisor; if it were, one would have $-E\cdot l=-(a+3)\ge 0$ whereas $l^2=1\ge 0$, a contradiction. Hence it follows that  $$h^2(T,L)=h^0(T,\omega_T\otimes L^{-1})=h^0(T,\mathcal{O}_T(-E)=0$$ and we obtain $$h^0(T,L)=g+\deg X.$$ Let $\mathcal{I}\subset\mathcal{H}_{d,g,5}$ be the locus consisting of  curves lying on a smooth del Pezzo surface. Let  $\h{D}_5$ be the space of smooth del Pezzo surfaces in $\PP^5$; $\dim\h{D}_5=35$. We have 
$$\dim\h{I}=\dim \h{D}_5+\dim |\mathcal{O}_T(a;b_1,\cdots ,b_4)|=35+g-1+\deg X<\Xx(d,g,5)$$
 if and only if \begin{equation}\label{nofull}d>\frac{1}{5}(3g+32).
 \end{equation}
\vnii (ii) An integral curve $X$ lying on a singular del Pezzo surface $T\subset\PP^5$ with isolated singularities is a specialization of smooth curves lying on a smooth del Pezzo; cf. \cite[Proposition 2.1]{edinburgh}
\end{remark}
\subsection{Smoothability of two nodal curves}
We prove the existence of some smoothable nodal curves for the next section.
In doing so, we need to use the case $n=5$ of the following lemma.
\begin{lemma}\label{rnc1}
Fix an integer $n\ge 3$ and a hyperplane $H\subset \PP^n$. Fix $A\subset H$ such that $\#A = n$ and $A$ spans $H$. Take $E\subseteq A$ such that
$\#E=3$, say $E=\{p_1,p_2,p_3\}$. Fix general lines $L_1$, $L_2$ and $L_3$ such that $p_i\in L_i$ for all $i$. Then there is a rational normal curve $C\subset\PP^n$ such that $A\subset C$ and each $L_i$ is the tangent line of $C$ at $p_i$.
\end{lemma}
\begin{proof}
Note that for a fixed $H$, all triples $(H,A,L_1)$ are projectively equivalent. 
The existence of a rational normal curve $C_1$ containing $A$ with the tangent to $L_1$ at $p_1$ is obvious; we may choose $L_1\nsubseteq H$. 
Assume, for the moment, the existence of a rational normal curve $C_2$ containing $A$
and tangent to $L_1$ at $p_1$ and to $L_2$ at $p_2$.

Recall that
$N_{C_2}={\Oo_{\PP^1}}(n+2)^{\oplus n-1}$ is a direct sum of  ($n-1$)  line bundles of degree $(n+2)$.  Set $Z:= (A\setminus \{p_1,p_2\})\cup \{2p_1\}\cup\{ 2p_2\}$. We have $\deg Z=n+2$
as an effective divisor on $C_2$, $h^1(C_2,N_{C_2}(-Z)) =0$ and $h^0(C_2,N_{C_2}(-Z)) =n-1$. Let $\Ee\subset\mathcal{H}_{n,0,n}$ be the closed subscheme of the Hilbert scheme $\mathcal{H}_{n,0,n}$ parametrizing smooth rational normal curve containing $Z$. 
We have $C_2\in \cal{E}\neq\emptyset.$
Since $h^1(C_2,N_{C_2}(-Z)) =0$, $C_2$ is a smooth point of $\Ee$ and $\dim _{C_2}\Ee = h^0(C_2,N_{C_2}(-Z)) =n-1.$

 Any $X\in \Ee$ contains $A$ and has tangent line $L_i$ at $p_i$, $i=1,2.$
Let $\mathbb{H}_3\subset\mathbb{G}(1,n)$ be the set of lines in $\PP^n$ containing $p_3$. Note that $\mathbb{H}_3\cong\PP^{n-1}$.
We have an injective map $\Ee\stackrel{\eta}{\to}\mathbb{H}_3$, sending $X\in\Ee$ to the tangent 
line $\eta(X)\in \mathbb{H}_3$ to $X$ at $p_3$. Since $\dim\Ee=\dim \mathbb{H}_3=n-1$, $\eta$ is dominant. Thus 
for a general $L_3\in \mathbb{H}_3$, there is a rational normal curve $C\subset\PP^n$ belonging to the family $\Ee$ with the desired property.
For the existence of $C_2$ we assumed tentatively in the beginning of the proof, we may use the same argument as above to get $C_2$ from $C_1$, which is easier and simpler.
\end{proof}
\begin{lemma}\label{1513} There is a component of $\HO{15,13,5}$ whose general element is ACM.
\end{lemma}
\begin{proof}
Fix a general hyperplane $H\subset \PP^5$, $A\subset H$ with $\#A =5$ spanning $H$ and take $E\subset A$ such that $\#E =3$. We take a rational normal curve $C\subset\PP^5$ satisfying Lemma \ref{rnc1}. Let $D\subset \PP^4=H$ be an extremal curve of degree $10$ and genus $9$. 
We may regard $D$ as a smooth element of a component of reducible Hilbert scheme $\HO{10,9,4}$ of the expected 
dimension $\Xx(10,9,4)=42$. Note that there is another component of $\HO{10,9,4}$ consisting of trigonal curves. Here we take the component other than the trigonal locus; cf. \cite[Theorem 2.2]{KK3}. Indeed $D\in|\Oo_{\FF_1}(4h+6f)|$ is $4$-gonal and $D$ has a plane model 
of degree $6$ with one node and $h^1(D, N_{D,H})=0$ \cite[Proposition 2.4]{cil}, \cite[Theorem 2.2]{KK3}. We may assume such $D$ contains $A$.
We have $h^0(D,\omega _D(-1)) =h^1(D,\Oo_D(1)) =3$. Since $E$ is also general in $D$, we have $h^0(D,\omega_D(-1)(-E)) =0$.
i.e. 
\begin{equation}\label{DE}
h^1(D,\Oo_D(1)(E))=0.
\end{equation} 
Set $Y:= C\cup D$. The reducible curve $Y$ is a connected nodal curve of degree $15$ and arithmetic genus $p_a(Y)=p_a(D)+p_a(C)+\deg(C\cap D)-1=13$. Obviously, $Y$ spans $\PP^5$. 
We want to 
prove that $Y$ is ACM and that $Y$ is smoothable. Indeed, if $Y$ is ACM and smoothable, it is a flat limit of elements of a component to which $Y$ belongs. Since $Y$ is connected, $h^1(\Ii_Y)=0$. Since $C$ is linearly normal in $\PP^5$ and  $D$ is linearly normal in $H$ and $A$ spans $H$, $Y$ is not an isomorphic linear projection of a reducible nodal curve spanning $\PP^6$ and hence $h^1(\Ii_Y(1)) =0$. Thus to prove that $Y$ is ACM it is sufficient to prove that $h^1(\Ii_Y(t)) =0$ for all $t\ge 2$.
\vni{\bf Claim 1:} $h^0(\Ii_Y(t))=0$ for all $t\ge 2$.
\vnii
{\it Proof of Claim 1:}
Consider the residual exact sequence of $H$:
\begin{equation}\label{eqhh2}
0\to \Ii_C(t-1) \to \Ii_Y(t) \to \Ii_{D,H}(t)\to 0
\end{equation}
Since $h^1(\Ii_C(t-1)) =0$, from the long cohomology exact sequence of \eqref{eqhh2}, it is sufficient to prove that
$h^1(H,\Ii_{D,H}(t)) =0$ for all $t\ge 2$. Note that $D\subset\PP^4$ is an extremal curve, hence is ACM and
$h^1(H,\Ii_{D,H}(t)) =0$ for all $t\ge 2$.\qed

\vni{\bf Claim 2:} $h^1(Y,N_Y)=0$ and $Y$ is smoothable.
\vni
{\it Proof of Claim 2:}
Let $N_Y$, $N_C$ and $N_D$ denote the normal bundles of $Y$, $C$ and $D$ in $\PP^5$. Let $N_{D,H}$ denote the normal bundle of $D$ in $H\cong\PP^4$. Note that $
N_D\cong N_{D,H}\oplus \Oo_D(1)$ and that $h^1(D,\Oo_D(1)) =3$. We use \cite[Theorem 4.1]{hh}.
Since $Y$ is nodal with $\mathrm{Sing}(Y)=(D\cap C)=A$, we have an exact sequence
\begin{equation}\label{eqhh3}
0\to N_Y \to N_C^{+}\oplus N_D^{+}\to \Oo_A \to 0
\end{equation}
where $N_C^{+}$ is obtained from $N_C$ by making $5$ positive elementary transformations, one for each $p\in A$
in the direction of the tangent line of $D$ at $p$, and $N_D^{+}$ is obtained from $N_D$ by making $5$ positive elementary transformations, one for each $p\in A$ 
in the direction
of the tangent line of $C$ at $p$; cf. \cite[eq. (1) p. 108]{hh}.  
We have $\mathrm{rank\,}N_Y=\mathrm{rank\,}N_C=\mathrm{rank\,}N_D=4$ and $\mathrm{rank\,}N_{D,H}=3$. 
Since $C$ is a rational normal curve in $\PP^5$, $N_C={\Oo_{\PP^1}}(7)^{\oplus 4}$ and $\deg(N_C) =28$.
 Hence $N_C^{+}$ is the direct sum of $4$ line bundles of degree at least $7$,
we get
the surjectivity of the restriction map $\rho: H^0(C,N_C^{+})\to H^0(A,\Oo_A)$ and $h^1(C,N_C^{+}) =0$.

Since $\rho$ is surjective and $h^1(C,N_C^{+}) =0$, from the long cohomology exact sequence of \eqref{eqhh3},
it is sufficient to show $h^1(D,N_D^{+}) =0$ to conclude $h^1(Y,N_Y)=0$.
By \cite[Theorem, Proposition 2.4]{cil}, we have $h^1(D,N_{D,H})=0$, hence $h^1(D,N_{D}) =h^1(N_{D,H}\oplus \Oo_D(1))=h^1(D,\Oo_D(1)) =3$. Recall \eqref{DE}, $h^1(D,\Oo_D(1)(E)) =0$.
Let $G_E$ be a vector bundle on $Y$ obtained from $N_{D,H}\oplus \Oo_D(1)$ making $3$ general positive elementary transformations, one for each point of $E$, thus $h^1(D,{G_E}_{|D})=0$  by $h^1(D,\Oo_D(1)(E)) =0$.
Note that the vector bundle  $N_D^+$ is obtained from $G_E$ making $2$ further positive elementary transformations, one for each point of $A\setminus E$. Hence $h^1(D,N_D^+)=0$ and
$Y$ is smoothable by \cite[Theorem 4.1]{hh}
\end{proof}
\begin{lemma}\label{1411} There is a component of $\HO{14,11,5}$ whose general element is ACM.
\end{lemma}
\begin{proof}
Take a rational normal curve $C\subset \PP^5$. Take a general hyperplane $H\subset \PP^5$ transversal to $C$. Set $S:=C\cap H$. Take $D'\subset H'\in{\PP^5}^*, H\neq H'$, $D'\subset H'$ is an extremal curve  with degree $9$ and genus $7$ in $\PP^4$, $S'\subset D'$, with $\#S'=5$ which is in general position. 
Choose  $\tau\in\Aut(\PP^4)$ such that $\tau (S')=S$. We have $\tau(D')=D$ contains 
$S$ and $C\cap D=S$. We may regard $D$ as a smooth element of irreducible Hilbert scheme $\HO{9,7,4}$ of the expected 
dimension $\Xx(9,7,4)=39$; cf. \cite[Theorem 2.1]{KKy2}. We have $D\in|\Oo_{\FF_1}(3h+6f)|$ is trigonal  and $h^1(D, N_{D,H})=0$ \cite[Proposition 2.4]{cil}.

The reducible curve $Y:= C\cup D$
is connected, nodal, $\deg Y=14$, $p_a(Y)=p_a(C)+p_a(D)-\deg (C\cap D)-1=11$. Moreover, $Y$ spans $\PP^5$. Since $Y$ is connected, $h^1(\Ii_Y)=0$. Since $C$ is linearly normal in $\PP^5$ and $D$ is linearly normal in $H$, $Y$ is not a isomorphic projection of a curve spanning $\PP^6$. Hence $h^1(\Ii_Y(1)) =0$. 

\vni{\bf  Claim 1:} $h^0(\Ii_Y(t))=0$ for all $t\ge 2$; the proof is same as the proof of Claim 1 in Lemma \ref{1513} using the exact sequence \eqref{eqhh2}.

\vni{\bf Claim 2:} $h^1(Y,N_Y)=0$ and $Y$ is smoothable.
\vnii
{\it Proof of Claim 2:} Claim 2 finishes the proof of Lemma \ref{1411} by semicontinuity for cohomology.
Let $N_Y$, $N_C$ and $N_D$ denote the normal bundle of $Y$, $C$ and $D$ in $\PP^5$. Let $N_{D,H}$ denote the normal bundle of $D$ in $H$. Note that $N_D\cong N_{D,H}\oplus \Oo_D(1)$ and that $h^1(D,\Oo_D(1)) =2$. Since $Y$ is nodal with $S =D\cap C$, we have an exact sequence (same as \eqref{eqhh3})
\begin{equation}\label{eqhh4}
0\to N_Y \to N_C^{+}\oplus N_D^{+}\to \Oo_S
\to 0,
\end{equation}
where $N_C^{+}$ is obtained from $N_C$ making $5$ positive elementary transformation, one for each $p\in S$, in the direction of the tangent line to $D$ at $p$, and $N_D^{+}$ is obtained from $N_D$ making $5$ positive elementary transformations, one for each $p\in S$  in the direction
of the tangent line to $C$ at $p$. \, Since $C$ is rational, $N_C$ splitts 
and
$N_C=\Oo_{\PP^1}(5+2)^{\oplus 4}$.  
Thus, $N_C^{+}$ is the direct sum of $4$ line bundles of degree at least $7$, $h^1(C,N_C^{+}) =0$ and we get
the surjectivity of the map $\rho: H^0(C,N_C^{+})\to H^0(S,\Oo_S)$ and $h^1(C,N_C^{+}) =0$. 
Since $\rho$ is surjective and $h^1(C,N_C^{+}) =0$, from the long cohomology exact sequence of \eqref{eqhh4},  it is sufficient to show  $h^1(D,N_D^{+}) =0$ to conclude $h^1(Y,N_Y)=0$. 

By \cite[Theorem, Proposition 2.4]{cil} or \cite{KKy2}, we have $h^1(D,N_{D,H})=0$, hence $h^1(D,N_{D}) =h^1(N_{D,H}\oplus \Oo_D(1))=h^1(D,\Oo_D(1)) =2$. Since $C$ is transversal to $H$, $\#S=5\ge 2$, we have $h^1(D,\Oo_D(1)(S)) =0$. Thus, $h^1(D,N_D^{+}) =0$ as desired. \end{proof}
\vspace{-12pt}
\section{Hilbert functions of smooth curves of degree $d=15$ in $\PP^5$ }
In this section we study Hilbert functions of smooth curves of degree $d=15$ in $\PP^5$.
Several important properties of $\HO{15,g,5}$ such as irreducibility, the description of a general fibre of the moduli map, the gonality of a general element in each irreducible components have been studied in 
 \cite{bumi24}, \cite{edinburgh}. Thus we {\it concentrate} on deriving the Hilbert functions.
We denote by $\HL{d,g,r}$ the subscheme of $\HO{d,g,r}$ consisting of components whose general element is linearly normal. 

\begin{theorem}\label{d=15}
\begin{itemize}
\item[\rm{(i)}] Every $X\in\HO{15,18,5}$ is ACM and has maximal rank, with the Hilbert function given by  \eqref{exthil}.
\item[\rm{(ii)}] $\HO{15,17,5}=\emptyset$.
 \item[\rm{(iii)}] $\HO{15,16,5}$ has three irreducible components $\Gamma_i$, $i=1,2,3$ described as follows.
\item[\rm{(1)}] A general $X\in\Gamma_i$, $i=1,2$ lies on a rational normal surface scroll:~ 
\[
h^0(\PP^5, \Ii_X(t))=\begin{cases}
0: ~t=1\\
6: ~t=2\\
28: ~t=3\\
\binom{t+5}{5} -15t+15: t\ge 4
\end{cases}
h^1(\Ii_X(t))=\begin{cases}
0: ~1\le t\le2\\
2: ~t=3\\
0: ~t\ge 4
\end{cases}
\]
\item[\rm{(2)}] A general $X\in\Gamma_3$ is ACM, which is a complete intersection of quintic surface and a cubic hypersurface; 
\[
h^0(\PP^5, \Ii_X(t))=\begin{cases}
0: ~t=1\\
5: ~t=2\\
\binom{t+5}{5} -15t+15: t\ge 3.
\end{cases}
\]
\item[\rm{(iv)}] A general $X\in\HO{15,15,5}$ is ACM, $h^0(\PP^5,\Ii_X(1))=0$ and $h^0(\Ii_X(t)) =\binom{t+5}{5} -15t+14$ for all $t\ge 2$. 
\item[\rm{(v)}] $\HO{15,14,5}=\HL{15,14,5}$ is reducible with two components $\Hh_i$  $(i=1,2)$ of the expected dimension. A general $X\in\Hh_i$ ($i=1,2$) is ACM. 
\item[\rm{(vi)}] $\HO{15,13,5}$ is reducible with two components $\h{H}_1$ and $\h{H}_2$:

$\h{H}_1=\HL{15,13,5}$ and $\dim \HL{15,13,5}=\h{X}(d,g,r)=66$.  A general $X\in\Hh_1$ is ACM and lies on a smooth rational surface of {degree  $10$.}

$\dim \h{H}_2=68>\h{X}(d,g,r),$ a general $X\in\h{H}_2$ is trigonal which is the image of external projection of extremal curves of degree $15$  in $\PP^6$.
For a general $X\in\Hh_2$, 
\[
h^0(\PP^5, \Ii_X(t))=\begin{cases}
0: ~t=1\\
\binom{t+5}{5} -15t+12: t\ge 2
\end{cases}
h^1(\Ii_X(t))=\begin{cases}
1: ~t=1\\
0: ~t\ge 2\\
\end{cases}
\]
\item[\rm{(vii)}] If $g\le 12$, $\HO{15,g,5}$ is irreducible. For a general $X\in\HO{15,g,5}$, 
$$h^0(\PP^5,\Ii_X(t))=\begin{cases}\binom{5+t}{t}-(15t+1-g)~ \mathrm{ if } ~ t\ge3\\
\max\{0,g-10\}~\mathrm { if } ~ t=2\\
0~\mathrm { if } ~t=1
\end{cases}$$
\end{itemize}
\end{theorem}
\begin{proof} The irreducibility statements and description of components of $\HO{15,g,5}$ are main results of 
\cite{edinburgh, bumi24}, which we include in the statement of the theorem for the convenience of readers. Therefore we determine $h^1(\PP^5,\Ii_X(t))$ and the Hilbert function $h^0(\PP^5,\Ii_X(t))$.

\vni
(i) $\HO{15,18,5}$ is irreducible and $4$-gonal by \cite[Prop. 6.2]{edinburgh}. 
Every $X\in\HO{15,18,5}$ is an extremal curve and is ACM by  Remark \ref{basic1}. In \eqref{exthil}, we have $h^1(X,\Oo_X(1))=8, h^1(X,\Oo_X(2))=2, h^1(X,\Oo_X(t))=0, \forall t\ge 3$ by routine computation. 
\vni
(ii) $\HO{15,17,5}=\emptyset$ by \cite[Prop. 6.1]{edinburgh}.

\vni
(iii) (1) \cite[Lemma 5.6]{edinburgh} gives the estimate of $h^1(\Ii_X(t))$ for $t\neq 2$. To see $h^1(\Ii_X(2)) =0$, it is enough to  check $h^0(\PP^5,\Ii_X(2))=6$. Fix any surface $S\supset X$ such that $\deg (S) \le 5$; cf. \cite[Theorem 5.1]{edinburgh}. Let $M\subset \PP^5$ be a quadric hypersurface containing $X$. If $S\nsubseteq M$ then $\deg (M\cap S)\le 10<\deg X$, thus 
$|\Ii_X(2)| =|\Ii_S(2)|$. To compute $\dim|\Ii_S(2)|$, we take a general hyperplane $H\subset \PP^5$. We may assume $S$ is a smooth quartic surface; cf. \cite[Theorem 5.1; (1), (2)]{edinburgh}. We set $R:= S\cap H$, a rational normal ACM curve in $H$, i.e. $h^1(H,\Ii_{R,H}(t)) =0$  and  $h^0(H,\Ii_{R,H}(t)) = \binom{4+t}{4}-(4t+1)$ $\forall t\in \NN$.
For $t\in\NN$, we consider the exact sequence 
\begin{equation}\label{eqtnew2}
0 \to \Ii_S(t-1) \to \Ii_S(t)\to \Ii_{R,H}(t) \to 0.
\end{equation} 
Since $S$ is ACM \cite[Theorem 1.3.3]{juan}, 
$h^1(\PP^5, \Ii_S(t-1)) =0$ for all $t\in\NN$ and 
$h^0(\PP^5, \Ii_S(1))=0$.  The long cohomology exact sequence of \eqref{eqtnew2}  gives
$h^0(\PP^5,\Ii_X(2))=h^0(\PP^5,\Ii_S(2)) = h^0(H,\Ii_{R,H}(2))=6$.
\vni
\quad(2)  For a general $X\in\Gamma_3$, the statement follows from \cite[Lemma 5.10 (ii)]{edinburgh}.

\vskip 2pt
\vni
(iv) The irreducibility of $\HO{15,15,5}$ is shown in \cite[Theorem 4.1]{edinburgh}. Since $\pi(15,6)=13 <g=15$, $X$ is linearly normal, $h^0(\PP^5,\Ii_X(1))=h^1(\PP^5,\Ii_X(1))=0$. Since $2\deg(X)>\deg(K_X)$, $h^1(X,\Oo_X(t)) =0$ for all $t\ge 2$ and hence 
$$h^0(\Ii_X(t)) =\binom{t+5}{5} -15t+14 \mathrm{~~ if ~ and ~ only ~ if ~~} h^1(\Ii_X(t))=0.$$ 

Let $H\subset \PP^5$ be a general hyperplane. 

\vni
\quad(iv-I) In this step we show $h^1(\Ii_X(2)) =0$. Assume $h^1(\Ii_X(2)) >0$, i.e. assume $h^0(\Ii_X(2))=h^0(\Oo_{\PP^5}(2))-h^0(X,\Oo_X(2))+h^1(\Ii_X(2))\ge 6$. Since $X$ is linearly normal, $h^0(\PP^5,\Ii_X(2)) = h^0(H,\Ii_{X\cap H,H}(2))\ge 6$ by Remark \ref{pre2}\,(iv).
By Remark \ref{aa1}\,(i), $X\cap H$ is contained in a rational normal curve $C\subset H$. Since 
$$\#(X\cap H)>2\deg C,$$ every quadric in $H$ containing $X\cap H$ 
contains $C$.
Since $h^0(H,\Ii_{C,H}(2)) =\binom{4+2}{2}-h^0(C,\Oo_C(2))=6$ and $C$ is completely cut out by quadrics, $C$ is the base locus of $|\Ii_{X\cap H,H}(2)|$. 
Thus the intersection of quadrics containing $X$ meets $H$ exactly in $C$. It then follows that  the intersection of quadrics containing $X$ is a surface $T$ whose general hyperplane section is a rational normal curve and therefore $\deg T=4$. However, this is not possible under our current situation $(d,g)=(15,15)$ by Remark \ref{cone1}-(0), (1b), (2c).

\vni
\quad(iv-II) Now assume $t>2$. By Remark \ref{pre2}\,(3), $h^1(H,\Ii_{X\cap H,H}(t)) =0$ for all $t\ge 4$. 
Assume that $h^1(H,\Ii_{X\cap H,H}(3))> 0$. By Lemma \ref{o1oo1}, $X\cap H\subset C$ where 

(i) $C\subset H$ is a rational normal curve or 

(ii) a linearly normal curve of arithmetic genus $1$. 

\vnii
If $\deg C=4$, we have $X\subset T$ with $\deg T=4$. Then by the same routine as we did in the latter part of the previous step (iv-I) above, we may conclude that this does not happen. 

The case $C$ being a linearly normal elliptic curve may occur.
Indeed, by 
 \cite[Theorem 4.1 \& the last part (Conclusion) of its proof]{edinburgh}, a general $X\in\HO{15,15,5}$ lies on a smooth del Pezzo surface $T\cong\PP^2_4$ and $X\in|(8;3,2^3)|$. Set $L:=(3; 1^4)=-\omega_T$.
\begin{align*}h^1(T,\Ii_{X,T}(3))&=h^1(T,\Oo_T(3L-X)) \\&=h^1(T,\Oo_T((1;0,1^3))=h^1(T,\omega_T\otimes(4;1,2^3)).
\end{align*} 
Since $\Oo_T(4;1,2^3)$ is nef and big \cite[Prop. 3.2, Th. 3.4]{sandra}, by Kawamata-Vieweg vanishing theorem, we have $h^1(T,\Ii_{X,T}(3))=0$. Thus,  the restriction map $H^0(T,\Oo_T(3))\stackrel{\rho_{T,X,3}}{\to} H^0(X,\Oo_X(3))$
is surjective. By \cite{laz}, $H^0(\PP^5, \Oo(3))\stackrel{\rho_{T,3}}{\to}  H^0(T,\Oo_T(3))$ is also surjective , thus $\rho_{X,3}=\rho_{T,X,3}\circ\rho_{T,3}: H^0(\PP^5, \Oo(3))\to H^0(X,\Oo_X(3))$ is surjective and we get $h^1(\Ii_X(3))=0$. 

Steps (iv-I) \& (iv-II), induction on $t$ and the long cohomology sequence of the sequence \eqref{eqppp2} give
$h^1(\Ii_X(t)) =0$, $\forall t\ge 1$.
 \vnii
(v) By \cite[Theorem 1.1]{bumi24}, $\HO{15,14,5}$ is reducible with two components of dimension $\Xx (15,14,5)$. Take $X\in\HO{15,14,5}$ which is general in a component of $\HO{15,14,5}$. Since $\pi(15,6)=13 <g=14$, $X$ is linearly normal and $h^1(\Ii_X(1))=0$. Since $2\deg(X)>\deg(K_X)$, $h^1(X,\Oo_X(t)) =0$ for all $t\ge 2$. Hence for $t\ge 2$, $h^0(\Ii_X(t)) =\binom{t+5}{5} -15t+13$ if and only if $h^1(\Ii_X(t))=0$. Let $H\subset \PP^5$ be a general hyperplane. 

\vni
\quad(v-I) We first show that $h^1(\Ii_X(2)) =0$.

Assume $e:= h^1(\Ii_X(2))>0$, i.e., assume $h^0(\Ii_X(2)) = 4+e\ge 5$. Since $X$ is non-degenerate and linearly normal, the restriction map $H^0(\Ii_X(2)) \to H^0(H,\Ii_{X\cap H,H}(2))$ is an isomorphism by Remark \ref{pre2}\,(iv).

\vnii
\noindent
\quad(v-I-i) If $e\ge 2$, we have  $h^0(\Ii_X(2)) \ge 6$.  We then may proceed exactly as in the previous part (iv-I),  to conclude that $X\subset T$, $\deg T=4$. However, this is not possible under our current situation $(d,g)=(15,14)$ by Remark \ref{cone1}-(0), (1b), (2c).
\vnii
\quad(v-I-ii)
If $e=1$, by  Remark \ref{aa1} (ii), $X\cap H$ is contained in a curve $C$ of degree $\le 5$ cut out by $|\Ii_{X\cap H,H}(2)|$. Varying $H$, we get that $X$ is contained in a surface $T$ with $\deg(T)\in \{4,5\}$. 
The case $\deg(T)=4$ can be ruled out again by the same way as the previous case $e\ge 2$. 

If $X\subset T\subset\PP^5$, $\deg T=5$,  we assume $T\cong\PP^2_4$ is smooth. 
For $(d,g)=(15,14)$, $L=(8;3^2,2,1)$ satisfies \eqref{delPC}. Thus we are sure that there is a smooth curve with $(d,g)=(15,14)$ lying on a smooth del Pezzo $T\subset\PP^5$. However, by \eqref{nofull}, the family of smooth
curves $X$ lying on a smooth del Pezzo does not constitute a full component.
However, 
it is possible that $X$ may lie on a quintic surface $T$, which can be either 
\begin{itemize}
\item[(a)] a cone over a smooth quintic elliptic curve in $\PP^4$

\item[(b)] a del Pezzo surface with possibly with finitely many isolated double points

\item[(c)] an image of a projection into $\PP^5$ of a surface $\hat{T}\subset\PP^6$ of minimal degree $5$
with centre of projection outside $\hat{T}$. 
\item[(d)] a cone over a rational quintic curve(either smooth or singular) in $\PP^4$.
\end{itemize}

\vnii By \cite[Proposition 2.1, Lemma 5.12]{edinburgh} and Remark \ref{delpezzo1}\,(ii), curves such as (a), (b) do not form a full component.
Case (c) is not possible since our curve $X\subset\PP^5$ is linearly normal. In the case
(d), $T$ is 
the image of a linear projection of a cone $\w{T}\subset\PP^6$  over a rational normal curve $\w{C}\subset\w{H}\cong\PP^5$ with center of projection $p\in\w{H}\setminus\w{C}$. 
This is not 
possible under the existence of a linearly normal  $X\subset T$. 

\vnii
\quad(v-II) By Remark \ref{pre2}, $h^1(H,\Ii_{X\cap H,H}(t)) =0$ for all $t\ge 4$.
For $t=3$, we argue $h^1(H,\Ii_{X\cap H,H}(3)) =0$.
Assume that $h^1(H,\Ii_{X\cap H,H}(3))> 0$. By Lemma \ref{o1oo1}, $Y\cap H\subset C$, where 
$C\subset H$ is a rational normal curve or a linearly normal curve of arithmetic genus $1$. If $\deg C=4$, we have $X\subset T$ with $\deg T=4$. Then by the same routine as we did in (iv-I) -- which we omit -- we may conclude that this does not occur. If $\deg C=5$, we are in a similar situation to the case (v-I-ii) ($e=1$). Therefore, our curve $X$ lies on a quintic surface and  conclude that such a family does not constitute a full component of $\HO{15,14,5}$ by the same argument.

Thus steps (v-I) \& (v-II), induction on $t$ and the long cohomology sequence of the sequence \eqref{eqppp2} gives
$h^1(\Ii_X(t)) =0$, $\forall t\ge 1$.
 \vnii
(vi) The reducibility of $\HO{15,13,5}$ was shown in \cite[Proposition 3.2]{edinburgh}. The proof of \cite[Proposition 3.2]{edinburgh} gives a description of a general $X\in \Hh_1$; $X$ is isomorphic to the normalization of a nodal plane degree $9$ curve with $15$ ordinary nodes as its only singularities. A general element of $\Hh_1$ is linearly normal, while
every element of the other $\Hh_2$ is an isomorphic  linear projection of an extremal curve $Y\subset \PP^6$ from some $o\in \PP^6$  not contained in the secant variety of $Y$.

Take a general $X\in \Hh_2\subset \Hh_{15,13,5}$.
Since $2\deg(X) >\deg K_X$, we have $h^1(\Oo_X(t)) =0$ for all $t\ge 2$. 
$X\subset\PP^5$ is an external projection of a smooth curve $Y\subset\PP^6$, $\deg Y=15$, which is an extremal curve $\pi(15,6)=13=g$ contained in a quintic surface $S_1\subset\PP^6$. 
Since $\pi(15,7)=10<g=13$, 
$h^0(\Oo_X(1)) =7$ and $h^1(\Ii_X(1))=1$. By an elementary computation for the degree and genus formula for a {\it smooth} curve $Y$ lying on a minimal degree surface in $\PP^6$--similar to \eqref{conevertex1}, \eqref{cone} (Remark \ref{cone1}) -- we may conclude $S_1$
is smooth, {\it not a cone} and $S_1\cong\FF_1$ by the general choice of $Y\subset\PP^6$.
By Lemma \ref{oo1}, we have $h^1(\Ii_X(t))=0$ for all $t\ge2$.

For a general $X\in\Hh_1$, we proceed as follows. 
Since $2\deg(X) >24$, we have $h^1(X, \Oo_X(t)) =0$ for all $t\ge 2$. By \cite[Prop. 3.2]{edinburgh}, $X$ is the normalization of an integral degree $9$ plane curve with $15$ nodes
as its only singularities together with the normalization map $X\to \PP^2$ induced by $|K_X(-1)|$. Let $\pi:  S:=\PP^2_{15}\to \PP^2$ be the blowing up of $\PP^2$ at $15$ general points. 
We have $\mathrm{Pic}(S) =\ZZ^{16}$ with a basis formed by $(1;0^{15}):= \pi^\ast (\Oo_{\PP^1}(1))$ and  $15$ exceptional divisors.
We have  
\vnii
\quad$\omega_S \cong \Oo_S((-3;-1^{15}))$, $\Oo_S((1;0^{15}))_{|X} \cong K_X(-1)$,  
$K_X\cong\Oo_S((6;1^{15}))_{|X}$ by adjunction and  
$|\Oo_X(1)|=|K_X-K_X(-1)|=|\Oo_S((5;1^{15}))_{|X}|$. 
The line bundle $L:=\Oo_S((5;1^{15}))$ on $\PP^2_{15}$ is very ample by \cite[Theorem 0]{Coppens}, which induces an embedding $S\stackrel{\cong}{\to} T\subset\PP^5$ with $\deg T=L^2=10$.
By Lemma \ref{1513}, a general $X\in\Hh_1$ is ACM. 
\vni
(vii) For $g\le 12$, we are in the Brill-Noether range, i.e. $\rho(d,g,r)=\rho(15,g,5)=60-5g\ge 0$. By \cite[Prop. 3.4, 3.5, 3.6]{edinburgh}, $\HO{15,g,5}$  is irreducible of the expected dimension with a unique component dominating $\Mm_g$. By \cite{l8} a general $X\in \Hh_{15,g,5}$ has maximal rank and 
the computation for $h^0(\PP^5, \Ii_X(t))$ easily follows. 
\end{proof}
\vspace{-18pt}
\section {Curves of degree $d=14$ }
 
 Starting from this section, we treat curves in $\PP^5$ of degree $d\le 14$. We 
prove the non-emptiness,  the irreducibility and compute the Hilbert function of a general element of $\Hh_{d,g,5}$.
 Several techniques the authors employed in \cite{bumi24}, \cite{edinburgh} for the  study of the case 
$d=15$  may well be adopted to lower degree cases $d\le 14$. 
Therefore, the proofs
will be brief and we skip some details for the sake of brevity. 
\vspace{-2pt}
\begin{remark}\label{trivia1}
\begin{itemize}
\item[(i)]
Given a smooth curve $X\subset\PP^r$ ($r\ge 2$) of genus $g$,  let $$\Aut(\PP^r)X:=\{Y\subset\PP^r | Y=\sigma (X) \mathrm{~ for ~some~ }\sigma\in \Aut (\PP^r)\}$$ be the family consisting of all curves $Y\subset \PP^r$ projectively equivalent to $X$. 
If $g\ge 2$,  $X$ has only finitely many automorphisms, thus $G:=\{h\in \Aut(\PP^r)\,|\, h(X)=X\}$ is a finite group.
$\Aut(\PP^r)X\cong\Aut(\PP^r)/G$ is irreducible  and quasi-projective such that $\Aut(\PP^r)X\subset\mu^{-1}\mu(X)$. We have $$\dim\Aut(\PP^r)X=\dim\Aut(\PP^r)=(r+1)^2-1. $$ 
\item[(ii)] Let $X\subset\PP^r,\,r\ge 2$ be a rational normal curve.  Consider the subgroup
$$H:=\{h\in\Aut(\PP^r)\,|\,h(X)=X\}\cong\Aut(\PP^1)\le \Aut(\PP^r).$$
We have
$\Aut(\PP^r)X\cong\Aut(\PP^r)/H\cong\Aut(\PP^r)/\Aut(\PP^1)$ and 
$$\dim\Aut(\PP^r)X=\dim\Aut(\PP^r)-\dim\Aut(\PP^1)=(r+1)^2-4.$$
\item[(iii)] Let $X\subset\PP^r$, $r\ge 2$ be a non-degenerate elliptic curve. Recall that any elliptic curve $X\subset\PP^r$ is sent into itself only by finitely many automorphism of $\PP^r$, thus $$\dim\Aut(\PP^r)X=\dim\Aut(\PP^r)=(r+1)^2-1.$$
\item[(iv)] 
Given a very ample and complete linear series $g^r_d,$ ($\,r\ge 4$) on a curve $X$, we denote  $\overset{o}{\mathbb{G}}(s,r)\subset {\mathbb{G}}(s,r)$ by the open and dense subset of the Grassmannian ${\mathbb{G}}(s,r)$ consisting of very ample subseries' $g^s_d\subset g^r_d$$\, \, ,3\le s<r$.
\end{itemize}
\end{remark}
\begin{remark}\label{trivial} 
When $d$ large enough compared with $g$, e.g. $d\ge 2g-7, g+r\le d, r\ge 3 $, $\HO{d,g,r}$ is irreducible and dominates $\h{M}_g$; \cite[Corollary 1.5]{PAMS} and \cite[Theorem 2.1]{PAMS}.
\end{remark}

\begin{proposition}\label{d14}
\begin{itemize}
\item[\rm{(i)}] $\HO{14,g,5}=\emptyset$ if $g\ge 16$.
\item[\rm{(ii)}] $\HO{14,15,5}=\HL{14,15,5}$ is reducible with two components of dimensions  more than $\Xx(d,g,r)$, 
$\mu^{-1}\mu(X)=\Aut(\PP^5)X$ for every smooth $X\in\HO{14,15,5}$. $X$ is ACM with the Hilbert function given
by \eqref{exthil}.
\item[\rm{(iii)}] $\HO{14,14,5}=\HL{14,14,5}$ is irreducible, 
$\mu^{-1}\mu(X)=\Aut(\PP^5)X$ for every smooth $X\in\HO{14,14,5}$.
\[
h^0(\PP^5, \Ii_X(t))=\begin{cases}
28: ~t=3\\
\binom{t+5}{5} -14t+13: t\neq 3
\end{cases}
h^1(\Ii_X(t))=\begin{cases}
1: ~t=3\\
0: ~t\neq 3\\
\end{cases}
\]
\item[\rm{(iv)}] $\HO{14,13,5}=\HL{14,13,5}$ is irreducible,
every smooth $X\in\HO{14,13,5}$ is ACM, 
$\mu^{-1}\mu(X)=\Aut(\PP^5)X$,
$h^0(\Ii_X(t))=\binom{t+5}{t}-(14t-12)$ for $t\ge 2$ and $\dim \mu(\HO{14,13,5})=26>\lambda(14,13,5).$
\item[\rm{(v)}] $\HO{14,12,5}=\HL{14,12,5}$ is reducible  with two components,
$\mu^{-1}\mu(X)=\Aut(\PP^5)X$ for every smooth $X\in\HO{14,12,5}$. A general $X\in\HO{14,12,5}$
is ACM with the Hilbert function $h^0(\Ii_X(t))=\binom{t+5}{5}-(14t-11), t\ge 2$.
\item[\rm{(vi)}] $\HO{14,11,5}$ is reducible with two components $\Hh_i, \,i=1,2$. $\h{H}_1=\HL{14,11,5}$ and the other one
$\h{H}_2\neq \HL{14,11,5}$. A general $X\in\HL{14,11,5}$ is ACM.
\[
\mu^{-1}\mu(X)=
\begin{cases}\Aut(\PP^5)X; \mathrm{for ~~ general~~} X\in\Hh_1\\
\overset{o}{\mathbb{G}}(5,6)\times\Aut(\PP^5)X;  \mathrm{for ~~ general~~} X\in\Hh_2.
\end{cases}
\]
\item[\rm{(vii)}] $\HO{14,g,5}$ is irreducible for $2\le g\le 10$. $\HO{14,g,5}=\HL{14,g,5}$ if $9\le g\le 10$  and  $\HL{14,g,5}=\emptyset$ if $g\le 8$. A general $X\in \HO{14,g,5}$ with $9\le g\le 10$ is ACM.
\[
\mu^{-1}\mu(X)=
\begin{cases}(X_4\setminus\Delta)\times
\Aut(\PP^5)X; g=10, 
\\
\Jac_{14}(X)\times\overset{o}{\mathbb{G}}(5,14-g)\times\Aut(\PP^5)X; 2\le g\le 9
\end{cases}
\] where $ \Delta=\{D\in X_4| \dim|\Ee-D|\ge 0, \,\Ee\in W^1_6(X)  \}.$
$h^0(\PP^5,\Ii_X(1))=0, h^0(\PP^5,\Ii_X(t)) = \max\{0,\tbinom{t+5}{5} -1-td+g\}$ for $t\ge 2$.

\end{itemize}
\end{proposition}

\begin{proof}
A non-degenerate integral curve $X\subset\PP^5$ of degree $d=14$ has genus $$g\le\pi(14,5)=15; ~\pi_1(14,5)=13.$$
Since $\pi(14,6)=11$, $\HO{14,g,5}=\HL{14,g,5}$ for $12\le g\le 15$.
\begin{itemize}
\item[(i)] If $g\ge 16$, $\HO{14,g,5}=\emptyset$.
\item[(ii)] If $g=15$, $X\in\Hh_{14,15,5}=\HL{14,15,5}$ is an extremal curve, lying on a surface of minimal degree.
Every extremal curve is ACM,  $h^1(\Ii_X(t))=0$ for all $t\ge 1$. 
\vnii
\quad(ii-a)
If  $X$ lies on a Veronese surface, $X$ is the image of a smooth plane septic under the $2$-tuple embedding. Let $\Hh_0\subset \HO{14,15,5}$ be the family consisting of such curves on Veronese surfaces. Set $$\Ff_0:=\{ ~|\Oo_C(2)|~ |~ [C]\in\PP(H^0(\PP^2,\Oo(7))/\Aut(\PP^2)\}\subset \Gg^5_{14}. $$ 
\vni
Note $\dim\Ff_0=\dim\PP(H^0(\PP^2,\Oo(7))-\dim\Aut(\PP^2)=27$, thus
\begin{equation*}
\dim\Hh_0=\dim\Ff_0+\dim\Aut(\PP^5)=62>\Xx(14,15,5)=56.
\end{equation*}
For a smooth curve $X\subset S\subset\PP^5$  of even degree $d=2a$
lying on a Veronese surface $S$,  we recall \cite[Prposition 2.4]{cil}
\begin{equation}\label{plane}h^1(N_{X,\PP^5})=3(g-d+5)-3a+9.
\end{equation}
Thus,
\hskip 20pt $h^0(N_{X,\PP^5})=\Xx(14,15,5)+h^1(N_{X,\PP^5})=\dim\Hh_0$
\vnii
and $\Hh_0$ is a generically reduced {\it component}.
We recall that a smooth plane curve of degree $a\ge 4$ has a {\it unique} $g^2_a$ and a {\it unique} very ample \& complete $g^5_{2a}$. Thus $\mu^{-1}\mu(X)=\Aut(\PP^5)X$. 

\vnii
\quad(ii-b)
Suppose $X$ lies on a smooth rational normal surface scroll $S\cong\FF_0$ or $S\cong\FF_2$. Solving \eqref{sd} and \eqref{sg}, we have $X\in|4H-2L|$ and $X$ is $4$-gonal. 

If $X\subset\FF_0\cong Q\subset\PP^5$,  $X\in |\Oo_Q(4,6)|$, $m(X)=4$ by  \eqref{F0}. Set $$\Ff_{4,0}:=\{g^5_{14}~ |~g^5_{14}=|\Oo_X(1)|,  [X]\in |\Oo_Q(4,6)|/\Aut(Q)\}\subset\Gg^5_{14}.$$ By \eqref{sdn},  we have $\dim|4H-2L|=34$. Note that  

$\dim\Ff_{4,0}= \dim|4H-2L|-\dim\Aut(\FF_0)=28>\dim\Ff_0.$ 
\vnii
By lower semi-continuity of gonality, $\Ff_0$ -- the one we described in part (ii-a) -- is not in the boundary of the irreducible locus  $\Ff_{4,0}$,

If $X\subset\FF_2$, $X\in |\Oo_{\FF_2}(4h+10f)|$, $m(X)=2$ by  \eqref{F2}. 
Set 
$$\Ff_{4,2}:=\{g^5_{14}~ |~g^5_{14}=|\Oo_X(1)|,  [X]\in |\Oo_{\FF_2}(4h+10f)|/\Aut(\FF_2)\}\subset\Gg^5_{14}.$$ Since
$\dim\Ff_{4,2}=|\Oo_{\FF_2}(4h+10f)|-\dim\Aut(\FF_2)=27=\dim\Ff_0,$  $\Ff_{4,2}$ is not in the boundary of the component $\Ff_0$ and vice versa.

Let $\Hh_{4,0}$ ($\Hh_{4,2}$ resp.) be the $\Aut(\PP^5)$-bundle over $\Ff_{4,0}$ ($\Ff_{4,2}$ resp.). By Proposition \ref{z5}, 
$g^1_4$ and $g^1_6$ on $X\in\Hh_{4,0}$ are unique and hence a very ample \& complete  $g^5_{14}=|2g^1_4+g^1_6|$ is  {\it unique}. For $X\in\Hh_{4,2}$, we may also conclude that there exists a {\it unique} complete and very ample $g^4_{15}$ which is of the form $|\Oo_{\FF_2}(4h+10f)\tensor\Oo_X|$.


We now claim that curves in $\Hh_{4,2}$ are specialization of curves in $\Hh_{4,0}$. This is known 
to be true in a more general context. In our specific situation, one compares curves $\hat{X}\subset\PP^3$ for each $X\in\Hh_{4,i}$, $i=0,2$, induced by the residual series of the {\it unique} $g^1_4$ with respect to $|\Oo_X(1)|$, i.e. $|\Oo_X(1)-g^1_4|$. For $X\in \Hh_{4,0}$, $|\Oo_X(1)-g^1_4|=|g^1_4+g^1_6|$ and hence $\hat{X}\subset\PP^3$ lies on a smooth 
quadric.  For $X\in \Hh_{4,2}$, $|\Oo_X(1)-g^1_4|=|h+2f|_{|X}$, we get $\hat{X}\subset\PP^3$ lies on a quadric cone.
The claim follows by \cite[Propositions 1.6 \& 2.1]{Zeuthen}.

\vnii
\quad (ii-c) By Remark \ref{cone1}(1b), 
$X$ does not lie on a cone over a rational quartic.

Since $X$ is an extremal curve,  the Hilbert function is given by \eqref{exthil}. By the uniqueness of $|\Oo_X(1)|$,
$\mu^{-1}\mu(X)=\Aut(\PP^5)X$ for every smooth $X\in\HO{14,15,5}$. From our dimension count, there is a component $\Hh_4$ with $\dim\Hh_4=\dim\Ff_{4,0}+\dim\Aut(\PP^5)\neq\dim\Hh_0$. 
\vnii
\item[(iii)]
$g=14$: Since $\pi_1(14,5)=13<g=14<\pi(14,5)$, $X$ lies on a quartic surface  $S\subset\PP^5$. $S$ cannot be a Veronese since $p_a(X)\neq \binom{7-1}{2}$ or a cone by Remark \ref{cone1}(b).
Thus, $S$ is smooth quartic surface scroll. By solving \eqref{sd}, \eqref{sg}, we get
$X\in|3H+2L|, |K_S+X-H|=|4L|$ and
$X$ is trigonal,
Thus a very ample and complete $\mathcal{D}=g^5_{14}$ on $X\in\HO{14,14,5}$ is {\it unique} which of the form  $\Dd=|K_X-4g^1_3|$.

The irreducibility of $\HO{14,14,5}$ follows from the irreducibility of $\Mm^1_{g,3}$. 
For $X\subset\FF_0$, the Maroni invariant $m(X)=6$ is the maximal one,  while if $X\subset\FF_2$, $m(X)=5$ by Lemma \ref{maroni}.
$\HO{14,14,5}$ dominates $\Mm^1_{g,3}\setminus\Mm^1_{g,2}$ and
$$\dim\HO{14,14,5}=\dim\Mm^1_{g,3}+\dim\Aut(\PP^5)=64>\Xx(14,14,5).$$
$X\in \HO{14,14,5}$ is not ACM; cf. \cite[Corollary 3.17]{he}.  
To compute the Hilbert function, we
set $T\cong\FF_0\cong\PP^1\times\PP^1\subset\PP^5$ and let
\vnii
$H^0(\PP^5, \Oo(t))\stackrel{\rho_{T,t}}{\to}H^0(T,\Oo_{T}(t))
\stackrel{\rho_{T,X,t}}{\to}H^0(X,\Oo_X(t))$ 
be  restriction maps; $\rho_{X,t}={\rho_{T,X,t}}\circ{\rho_{T,t}}$ \& $\rho_{T,t}$ is surjective for $t\ge 1$ by \cite{laz}. 
Since $X\in|\Oo_{\PP^1\times\PP^1}(3,8)|$, 
\vnii
$h^1(T,\Ii_{X,T}(t))=
h^1(T,\Oo(tH-X))=
h^1(\PP^1\times\PP^1,\Oo(t-3,2t-8))>0$ only if 
$t=3$. Thus $h^1(\Ii_X(3))=h^1(\PP^1\times\PP^1,\Oo(0,-2))=1$, $h^0(\Ii_X(3))=28$.
For $t\ne 3$, $\rho_{X,t}$ is surjective, $h^1(\Ii_X(t))=0$ and $h^0(\Ii_X(t))=\binom{t+5}{5}-(14t-13)$. 
\vni
\item[(iv)] $g=13$: The irreducibility is a special case of \cite[Theorem 3.4]{speciality}. A general $X\in\Hh_{14,13,5}$ lies on a smooth del Pezzo surface 
$T\subset\PP^5$,  $X\in |(8;3^2,2^2)|$, $|K_X(-1)|=g^3_{10}$ is birationally very ample, 
$X^\vee\in |\Oo_{\PP^1\times\PP^1}(5,5)|$ singular with $3$-nodes; cf. \cite[proof of Theorem 3.4]{speciality}. $X$ has at least two $g^1_5$'s cut out by rulings of 
$\PP^1\times\PP^1$, say $M, N$ with $M\otimes N=|K_X(-1)|$. 

\vnii Claim: $|K_X(-1)|$ is the only {\it birationally very ample} $g^3_{10}$ on $X$:  Let $R$ be a $g^3_{10}$, which is necessarily birationally very ample by Remark \ref{bih},  base-point-free and complete by Castelnuovo genus bound. Choose $s_1, s_2\in H^0(X,M)$ independent sections and consider the cup-product map 

\hskip 40pt $H^0(X,M)\otimes H^0(X,R)\stackrel{\eta}{\to} H^0(X,R\otimes M).$

\vnii
Since $R\otimes M$ is birationally very ample, $h^0(X,R\otimes M)\le 7$ by Castelnuovo genus bound.  By the base-point-free pencil trick \cite[page 126]{ACGH}, $\dim\ker\eta=h^0(X,R\otimes M^{-1})\ge 1$. Thus
$|R\otimes M^{-1}|\neq\emptyset.$ Likewise, $|R\otimes N^{-1}|\neq\emptyset$, thus $R=M\otimes N=|K_X(-1)|$ and the Claim follows.  
Passing to residual series, $|\Oo_X(1)|$ is the {\it unique} $g^5_{14}$ on $X$ thus
$\mu^{-1}\mu (X)=\Aut(\PP^5)X$.

Since $X$ is linearly normal, $h^1(\Ii_X(1))=0$ . Recall $X\in |(8;3^2,2^2)|$ and let $T\cong\PP^2_4\subset\PP^5$. 
\vnii
\quad(I) First we show $h^1(\Ii_X(2))=0$. By \cite{laz}, $h^1(\Ii_T(t))=0$, $t\ge 2$ and $H^0(\PP^5, \Oo(t))\stackrel{\rho_{T,t}}{\to} H^0(T,\Oo_T(t))$ is surjective. Set $L:=(3;1^4)=-\omega_T$. 
We have 
\begin{align*}h^1(T,\Ii_{X,T}(t))&=h^1(T,\Oo_T(tL-X)) \\&=h^1(T,\Oo_T(((3t-8);(t-3)^2,(t-2)^2)).
\end{align*} 
For $t=2$,  $h^1(T,\Ii_{X,T}(2))=h^1(T,\Oo_T(-2;-1^2,0^2))=0$; note that  $\Oo_T(2;1^2,0^2)$ is nef and big \cite[Prop. 3.2, Th. 3.4]{sandra}, thus Kawamata-Vieweg vanishing can be applied. Thus, $\rho_{{X,T},2}$ is surjective and hence $\rho_{X,2}=\rho_{{T,X},2}\circ\rho_{T,2}$ is surjective.

\vnii
\quad(II)
For $t=3$,  $h^1(T,\Ii_{X,T}(3))=h^1(T,\Oo_T(1;0^2,1^2))=h^1(T,\omega_T\otimes(4;1^2,2^2))$.
Since $\Oo_T(4;1^2,2^2)$ is nef and big, $h^1(T,\Ii_{X,T}(3))=0$. Thus, $X$ is ACM by Remark \ref{pre2} (iii)
\&
$h^0(\Ii_X(t))=\binom{t+5}{t}-(14t-12)$ for $t\ge 2$.

We have $\mu(\HO{14,13,5})\subset\Mm^1_{g,5}$. However $\mu$ does not dominate
$\Mm^1_{g,5}$. $\mu(\HO{14,13,5})$ is the locus of $5$-gonal curves 
with another $g^1_5$. To estimate the dimension of $\mu(\HO{14,13,5})$, we use the description of the dual curve $X^\vee\in|\Oo_{\PP^1\times\PP^1}(5,5)|=:\Ll$ with $\delta=3$ nodes, 
$$ \mu(\HO{14,13,5})\stackrel{bir}{\cong}\Sigma_{\Ll,\delta}/\Aut{(\PP^1\times\PP^1)},$$
where $\Sigma_{\Ll,\delta}$ is the Severi variety of integral curves in $\Ll$ on $\FF_0$ with $\delta$-nodes. Thus we have

$\dim \mu(\HO{14,13,5})=\dim\Ll-\delta-\Aut(\FF_0)=26>\lambda(14,13,5).$
\vnii
\item[(v)] $g=12:$
The reducibility of $\HO{14,12,5}$ is known by \cite[Theorem 3.4(ii)]{lengthy}, which deserves some explanation in relation to the moduli map. It was shown that $\HL{14,12,5}$ is reducible with
two components of the minimal possible dimension, say $\Hh_4$ formed by smooth 
curves $X\subset\PP^5$ induced by the very ample $|K_X(-2g^1_4)|$ on a general $4$-gonal curve. Since there is only one $g^1_4$ on a general member in $\Mm^1_{g,4}$ and $|\Oo_X(1)|=|K_X(-2g^1_4)|$ we have  $\mu^{-1}\mu(X)=\Aut(\PP^5)X$ for  a general $X\in\Hh_4$ and $\mu$ dominates $\Mm^1_{g,4}$. We have $\dim\Hh_4=\Xx(14,g,5)=\dim\Mm^1_{g,4}+\dim\Aut(\PP^5)=62$. The other component $\Hh_{\Sigma_{8,12}}$ is the one induced by the very ample linear series residual to the hyperplane series of a singular plane curve of degree $8$ with $\delta=9$ nodes. By  \cite[Theorem 3.9]{treger} the normalization $X$ of a general element in the Severi variety ${\Sigma_{8,12}}$ has a unique \bpf $g^2_8$, hence the very ample $g^5_{14}=|K_X(-g^2_8)|$ is unique and we have  $\mu^{-1}\mu(X)=\Aut(\PP^5)X$ for  a general $X\in\Hh_{\Sigma_{8,12}}$.
We now take a general $X\in\Hh_4$ or $X\in\Hh_{\Sigma_{8,12}}$.

\vnii
\quad Step (I): In this step we show $h^1(\Ii_X(2)) =0$. Assume $h^1(\Ii_X(2)) >0$, i.e. assume $h^0(\Ii_X(2))=h^0(\Oo_{\PP^5}(2))-h^0(X,\Oo_X(2))+h^1(\Ii_X(2))\ge 5$. By Remark \ref{pre2}\,(iv), $h^0(H,\Ii_{X\cap H,H}(2))\ge 5$.
By Remark \ref{aa1}\,(ii), either (i) $X\cap H$ is contained in a rational normal curve $C\subset H$ or (ii) an elliptic normal curve $D\subset H$. 

\vnii
\quad(i) Assume $X\cap H\subset C$, a rational normal curve. Using exactly the same argument employed in the proof (iv-I) in Theorem \ref{d=15},
we deduce that $X$ lies on a quartic surface $T\subset \PP^5$.  However, we argue that $\deg T=4$ is not
possible as follows.
\begin{itemize}
\item[(a)]
Since $\deg(X)=14$ and $g\neq\binom{6}{2}$, $T$ is not a Veronese surface. 
\item[(b)] $T$ is not a cone over a rational normal curve by Remark \ref{cone1} (1b). 
\item[(c)]
Assume that $T\cong Q\cong\FF_0$ embedded by the complete linear system $|\Oo_Q(1,2)|$. We take  $(a,b)\in \NN\times\NN$ such that $X\in |\Oo_Q(a,b)|$. We have $\deg(X)=14 =2a+b$,  
$g=12=(a-1)(b-1)$ and hence $(a,b)=(5,4)$. 
Recall that  $\h{H}_4$ dominates $\Mm^1_{g.4}$. Hence a general $X\in\Hh_4$ has no base-point-free $g^1_5$ by general theory; a general $k$-gonal curve has no base-point-free and complete $g^1_{k+1}$ as long as $\rho({k+1,g,1})<0$; cf. \cite[Theorem 2.6]{AC1}.
If  $X\in\Hh_{\Sigma_{8,12}}$, 
a general 
$X\in\Hh_{\Sigma_{8,12}}$ has no $g^1_4$ by \cite{Coppens0}.
\item[(d)]
Assume $T\cong \FF_2$,
 and set $X\in |\Oo_{\FF_2}(ah+bf)|$ with $h^2=-2$. 
Solving \eqref{F2}, we get $(a,b)=(5,9)$.
Since $X\cdot f=5$, $X$ has a base-point-free $g^1_5$, contrary to the fact that a general $X\in \h{H}_4$  or $X\in\Hh_{\Sigma_{8,12}}$ does not have a base-point-free $g^1_5$ as we noted in  (c) above.
\end{itemize}

\vnii
\quad (ii) Assume $X\cap H\subset D$, a elliptic normal curve. In this case we may deduce $X\subset T$, $\deg T=5.$ We assume $T\cong\PP^2_4$ is smooth. For $(d,g)=(14,12)$, $L=(10;4^4)$ satisfies \eqref{delPC}. Thus there is a smooth curve with $(d,g)=(14,12)$ lying on a smooth del Pezzo $T\subset\PP^5$. However, by \eqref{nofull}, the family of smooth
curves $X$ lying on a smooth del Pezzo does not constitute a full component.
Thus Step (I) is done.
\noindent

\vnii
\quad Step (II): We show $h^1(H,\Ii_{X\cap H,H}(3))=0$:

\noindent
Assume $h^1(H,\Ii_{X\cap H,H}(3)) >0$. Then by Lemma \ref{o1oo1}, $X\cap H$ lies on a rational normal curve $C\subset H$. By repeating the same routine as in Step (I)\,(i) above,  $X$ lies on a surface of minimal degree and deduce $h^1(H,\Ii_{X\cap H, H}(3))=0$. 

\vnii
\quad Step (III): By Step (II) \& Remark \ref{pre2}\,(ii), we have $h^1(H,\Ii_{X\cap H,H}(t)) =0$ for all $t\ge 3$. Thus Step (I), induction on $t$ and the long cohomology exact sequence of the short exact sequence \eqref{eqppp2} give
$h^1(\Ii_X(t)) =0\, \,\forall t\in\NN$.

 \vni

\item [(vi)] If $g=11$, $\HL{14,11,5}$ is irreducible of the expected dimension dominating the irreducible locus $\Mm^1_{g,6}$ and $|K_X(-1)|=g^1_6$ for general $X\in\HL{14,11,5}$; \cite[Theorem 2.3]{JPAA} and \cite[Theorem 2.1]{lengthy}. Since $g^1_6$ is unique on a general member in $\Mm^1_{g,6}$, $\mu^{-1}\mu (X)=\Aut(\PP^5)X$ for a general  $X\in\HL{14,11,5}$.
Assume that there is a smooth, non-linearly normal $X\in \HO{14,11,5}$, i.e.  $\dim|\Oo_X(1)|\ge 6$. By Clifford's theorem, $|K_X(-1)|=g^2_6$. Since $g=11>\binom{6-1}{2}$,
$|K_X(-1)|=g^2_6$ is compounded. $X$ cannot be bielliptic by Remark \ref{bih}.
Thus $|K_X(-1)|=2g^1_3$. Since $|K_X(-2g^1_3)|=g^6_{14}$ is very ample for every trigonal curve of genus $g\ge 10$, we have an irreducible 
family $\Gg_3\subset
\Gg^5_{14}$ consisting of incomplete very ample $g^5_{14}$'s which surjects onto $\Mm^1_{g,3}\setminus\Mm^1_{g,2}$ with 

$\dim\Gg_3=\dim\Mm^1_{g,3}+\dim\mathbb{G}(5,6)=\lambda(d,g,5)=29.$ 

\vnii
Thus we have the extra component $\Hh_3\neq\HL{14,11,5}$ of the expected dimension. For a general $X\in\Hh_3$, 
$\mu^{-1}\mu (X)=\overset{o}{\mathbb{G}}(5,6)\times\Aut(\PP^5)X.$

\vnii

A general $X\in\HL{14,11,5}$ is ACM by Lemma \ref{1411}.

A general $X\in \Hh_3$ is trigonal, $\dim|\Oo_X(1)|=6$. $X$ is the image of an external projection 
of curve $Y\subset\PP^6$ which is extremal since $\pi(14,6)=11=g$. By Lemma \ref{oo1}, we have $h^1(\Ii_X(t))=0$ for all $t\ge2$ and $h^1(\Ii_X(1))=1$.
\item[(vii)]
If $g\le 10$, $\rho (14,g,5)\ge0$ and $\Hh_{14,g,5}$ has a unique component of the expected dimension dominating $\Mm_g$. 
\vnii
For the irreducibility and others, we first consider the case $g=10$.

\vnii
\quad (a) $g=10$: By \cite[Theorem 2.2]{JPAA},  $\HL{d,g,r}$ is irreducible of the expected dimension. Suppose that there is a component $\Hh\subset\HO{d,g,r}$ consisting of non-linearly normal curves. For a general $X\in \Hh$,
$|K_X(-1)|=g^1_4$
by Clifford's theorem. However, the family arising this way has dimension $\dim\Mm^1_{g,4}+\dim\mathbb{G}(5,6)+\Aut(\PP^5)=64<\Xx(d,g,r)=66$, which does not constitute a full component. Thus $\HO{d,g,r}=\HL{d,g,r}$ is irreducible.

We take a general $X\in\HO{14,10,5}$. Since $\rho(14,10,5)>0$, $\HO{14,10,5}$ dominates $\Mm_g=\Mm^1_{g,6}$. 
Let $X_K\subset\PP^9$ be the canonical curve of a non-hyperelliptic curve $X$ of genus $g=10$. For any choice $p+q+r+s\in X_4\setminus\Delta$, where 
$\Delta=\{D\in X_4 \,| \,\dim|\Ee-D|\ge 0, \,\Ee\in W^1_6(X)  \}$, 
the projection $\pi_L: X_K\to \pi_L(X_K)\subset\PP^5$ with center at $L=\bar{pqrs}$ is an embedding onto a linearly normal curve with $\deg\,\pi_L(X_K)=\deg\, K_X-4=14$. Conversely, any linearly normal, non-degenerate smooth curve 
with $(d,g)=(14,10)$ in $\PP^5$ is obtained this way.
Thus for a general $X\in \HL{14,10,5}$,
$\mu^{-1}\mu(X)=(X_4\setminus\Delta)\times
\Aut(\PP^5)X$.

By a theorem of Max Noether \cite[p. 117]{ACGH}, the canonical curve of a non-hyperelliptic curve is ACM.  Therefore a successive general projection of a canonical curve onto $X\subset \PP^5$ is ACM by an easy modification of Lemma \ref{oo1} using \cite{l8}.

\vnii
\quad
(b) $g\le 9$: $\HO{d,g,r}$ is irreducible by Remark \ref{trivial}. 
Recall that if  $d\ge g+5$, a curve of genus $g\ge 2$ has a very ample and non-special $|D|=g^5_d$ by a theorem of Halphen. The uniqueness of a dominating component $\Hh_0$ of $\HO{d,g,r}$ 
implies that for a general $X\in\Hh_0$, $|\Oo_X(1)|$ is non-special.  Thus in the range $2\le g\le 9$ and $d=14$, the irreducibility of $\HO{d,g,5}$ gives the following description of the fibre of the moduli map for general $X\in\HO{d,g,5}$;
$
\mu^{-1}\mu(X)\cong
\Jac_{14}(X)\times\overset{o}{\mathbb{G}}(5,14-g)\times\Aut(\PP^5)X; \, 2\le g\le 9.
$
\vnii
Since $\rho(14,g,5)>0$, 
the restriction map $H^0(\PP^5, \Oo(t)\to H^0(X, \Oo_X(t))$ is of maximal rank by \cite{l8}
and the Hilbert function is given by \eqref{hil} for all $t\ge 2$; cf. Remark \ref{lars}. A general $X\in\HL{14,g,5}$ 
is ACM if $g=9$.
\end{itemize}\vspace*{-\baselineskip}
\end{proof}
\begin{remark} As we have seen in the proof of Proposition \ref{d14}(iii), every  smooth curve $X\subset\PP^5$ with $(d,g)=(14,14)$ is trigonal and the embedding into $\PP^5$ is given by $|K_X-4g^1_3|$. However  $|K_X-4g^1_3|$ may not be very ample for every trigonal curve of genus $g=14$. 
The only  trigonal curve  such that  $|K_X-4g^1_3|$ is not very ample is the curve lying on a rational normal cone with a double point at the vertex of the cone: Choose $\w{X}\in|\Oo_{\FF_4}(3h+14f)|$. $|K_{\FF_4}+\w{X}-4f|=|h+4f|$ is not very ample and its restriction to $\w{X}$ is 
not very ample inducing a birational morphism $\w{X}\to X\subset S_0\subset\PP^5$ where 
$X$ has a double point at the vertex of the cone $S_0$; $\w{X}\cdot (h+4f)=2$.
One may argue that  the only trigonal curve $\w{X}$ such that $|K_{\w{X}}-4g^1_3|$ is not very ample is of this type, i.e. $\w{X}\in|\Oo_{\FF_4}(3h+14f)|$, which is left to readers. 
\end{remark}
\vspace{-12pt}
\section {Curves of degree $d=13$ in $\PP^5$ }
\begin{proposition}\label{d13}
\begin{itemize}
\item[\rm{(i)}] $\HO{13,g,5}=\emptyset$ if $g\ge 13$.
\item[\rm{(ii)}] $\HO{13,12,5}=\HL{13,12,5}$ is reducible with two components. One component consists of trigonal curves with every possible Maroni invariants. The other one consists of $4$-gonal curves. $\mu^{-1}\mu(X)=\Aut(\PP^5)X$ for every smooth  $X\in\HO{13,g,5}$ and $X$ is ACM. 
\item[\rm{(iii)}] $\HO{13,11,5}=\HL{13,11,5}$ is irreducible.
$\mu^{-1}\mu(X)=\Aut(\PP^5)X$ for  a general $X\in\HO{13,11,5}$  and $X$ is ACM. 
\item[\rm{(iv)}] $\HO{13,10,5}=\HL{13,10,5}$ is irreducible. 
$\mu^{-1}\mu(X)=\Aut(\PP^5)X$ for  a general $X\in\HO{13,10,5}$ and $X$ is ACM.

\item[\rm{(v)}] $\HO{13,9,5}=\HL{13,9,5}$ is irreducible. For general $X\in\HO{13,9,5}$, $X$ is ACM, 
$\mu^{-1}\mu(X)=X_3\times
\Aut(\PP^5)X$. 

\item[\rm{(vi)}] $\HO{13,g,5}$ is irreducible and $\HL{13,g,5}=\emptyset$ if $g\le 8$. For a general $X\in\HO{13,g,5}$, 

$\mu^{-1}\mu(X)=\Jac_{13}(X)\times\overset{o}{\mathbb{G}}(5,13-g)\times\Aut(\PP^5)X; 2\le g\le 8.$

\vnii
A general $X\in\HO{13,g,5}$ is ACM for $g=8$.  For a general $X\in\HO{13,g,5}$ with $g\le 7$, we have $h^1(\Ii_X(t))=0$ $\forall t\ge2$  and $h^1(\Ii_X(1))=8-g$.
\end{itemize}
\end{proposition}
\begin{proof}
\begin{itemize}
\item[(i)] Since $\pi(13,5)=12$, $\HO{13,g,5}=\emptyset$ if $g\ge 13$.
\item[(ii)] If $g=12$, $X\in\Hh_{13,12,5}$ is an extremal curve lying on a quartic.  $X$ does not lie on a Veronese surface since $\deg X=13$ is odd. 

\item[(a)]
Suppose $X$ lies on a cone $S_0$ over a rational quartic in $\PP^4$. By Remark \ref{cone1}\,(b), we have $k=3, d=13=(r-1)k+m, m=1$. Let $\w{X}$ be the strict transformation of a smooth $X\subset S_0\subset\PP^5$ under the desingularization $\FF_4\stackrel{|h+4f|}{\to} S_0$. Put $\w{X}\in|\Oo_{\FF_4}(ah+bf)|$. We have
$d=\deg X= (ah+bf)\cdot(h+4f)=b$ and $K_{\FF_4}=|-2h-6f|$, thus by adjunction, $g=\left( a-1 \right)  \left( b-2\,a-1 \right) $, $\w{X}\in|\Oo_{\FF_4}(3h+13f)|$ and $|f|_{|\w{X}}=g^1_3$.

We now compute the Maroni invariant of $X\subset S_0\subset\PP^5$. Recall that for each $t\ge 1$, $a, b>0$,
$h^0(\FF_4,(-ah+(-b+t)f)) = 0,$ 
$h^0(\FF_4,\Oo_{\FF_4}(tf)) =t+1,$
 $h^1(\FF_4,\Oo_{\FF_4}(tf)) =0.$
From the long cohomology sequence of the exact sequence 
\begin{equation}\label{eqa23}
0\to \Oo_{\FF_4}(-ah+(-b+t)f)\to \Oo_{\FF_4}(tf) \to \Oo_{\w{X}}(tf)\to 0
\end{equation}
$h^0(\w{X},\Oo_{\w{X}}(tf)) =t+1$ if and only if $h^1(\FF_4,(-ah+(-b+t)f)) =0.$ By duality, $h^1(\FF_4,(-ah+(-b+t)f)) =h^1(\FF_4,((a-2)h+(b-6-t)f)) =0.$
Since $h^1(\Oo_{\FF_e}(h+mf))=|\min(0, m-e+1)|, $ by taking $a=3, b=13, e=4, m=b-6-t$ 
we have $h^1(\FF_4,(h+(7-t)f)) =0$ if and only if  $t\le 4$. Thus $m(X)=3.$

\item[(b)]

Suppose $X\subset S$ ($S\cong\FF_0\cong S_{2,2}$ or $S\cong\FF_2\cong S_{1,3}$). By solving \eqref{sd}, \eqref{sg} and by \eqref{sdn}, we have $X\in \Mm\cup\Nn$ where
$$\Mm:=|3H+L|,
\Nn:=|4H-3L|, \dim\Mm=31, \dim\Nn=29.$$
For $i=0,2$, we set $\Ff_{3,i}, \Ff_{4,i}\subset\Gg^5_{13}$ as
\begin{align*}
&\Ff_{3,i}:=\{|\Oo_X(1)| : X\in \Mm, X\subset\FF_i\}, \Ff_{4,i}:=\{|\Oo_X(1)| : X\in\Nn, X\subset \FF_i\} ~~~\&\\
&\Ff_{3,4}:=\{|\Oo_X(1)| : X\subset S_0, X\cong\w{X}\in|3h+13f| ~\mathrm{on} ~\FF_4\}\subset\Gg^5_{13}.
\end{align*}
\noindent
Since $\dim\Aut(\FF_0)=6$, $\dim\Aut(\FF_2)=7$, $\dim\Aut(\FF_4)=9$, we have
\vnii
$\dim\Ff_{3,i}=\dim\Mm-\dim\Aut(\FF_i)>\lambda(13,12,5)=21 \textrm{},$
\vnii
$\dim\Ff_{3,4}=\dim|\Oo_{\FF_4}(3h+13f)|-\dim\Aut(\FF_4)>\lambda(13,12,5) \textrm{},$
\vnii
$\dim\Ff_{4,i}=\dim\Nn-\dim\Aut(\FF_i) >\lambda(13,12,5).$
\vnii
\vnii
By Lemma \ref{maroni} and \eqref{F0},  in case $X\in\Mm$, we have 
\vnii
$X\in|\Oo_Q(3,7)|, m(X)=5$;  if $X\subset\FF_0,$
\vnii
$X\in|\Oo_{\FF_2}(3h+10f)|, m(X)=4$; if $ X\subset\FF_2$,
\vnii
$\w{X}\in|\Oo_{\FF_4}(3h+13f)|, m(X)=3 \mathrm{~ if ~} \w{X}\subset\FF_4,$ i.e. when 
$X$ lies on a cone. 
In general, the Maroni invariant of a curve of genus $g$ is in the range

\centerline{$(g-4)/3\le m\le (g-2)/2.$ }

\noindent
In our case $g=12$, $8/3<3\le m\le  5$, thus every possible Maroni invariant is realized by curves in $\HL{13,12,5}$ given by the linear series in the family $\Ff_{3,0}\cup\Ff_{3,2}\cup\Ff_{3,4}$. 

\vni
If $X\in\Nn$, by Lemma \ref{maroni} and \eqref{F0}, we have 
\begin{equation}
\begin{cases}
X\in|\Oo_Q(4,5)|,  m(X)=3; X\subset\FF_0,\\
X\in|\Oo_{\FF_2}(4h+9f)|, m(X)=1;  X\subset\FF_2.\label{f4}
\end{cases}
\end{equation}
\item[(c)] Let $X$ be any trigonal curve of genus $g=12$. We claim that 
$|K_X-3g^1_3|$ is \bpf and very ample: 
Suppose $|K_X-3g^1_3|$ has a base locus $\Delta, \deg\Delta =\delta\ge 1$. The \bpf ~$g^5_{13-\delta}$ is not  birationally very ample 
by the \cgb; $\pi(13-\delta, 5)\le 10$. Hence the compounded $g^5_{13-\delta}$ induces double covering onto an elliptic curve and $X$ has a \bpf $g^1_4$, a contradiction by Castelnuovo-Severi inequality. Thus, $\delta =0$ and $|K_X-3g^1_3|=g^5_{13}$ is 
birationally very ample since $13$ is prime. Furthermore by the \cgb, we have $p_a(X)=12$, hence the image curve in $\PP^5$ induced by
$|K_X-3g^1_3|$ is an extremal curve (smooth). Thus $|K_X-3g^1_3|$ is very ample. 
\vnii

Conversely we claim that every  (complete)  $g^5_{13}$ on a trigonal curve is of the form $|K_X-3g^1_3|$. Set
$\Ee:=|K_X-g^5_{13}|=g^3_9$. If $\Ee$ is birationally very ample, $
\Ee$ is \bpf and  induces an 
extremal curve in $\PP^3$ lying on a quadric surface of type $(4,5)$. Hence $X$ has a \bpf $g^1_4$, which is impossible by Castelnuovo-Severi inequality. Thus $\Ee$ is compounded, \bpf since a compounded $\Ee=g^3_9$ with non-empty base locus is bielliptic which is impossible by Castelnuvo-Severi inequality.
Thus we have $|K_X-g^5_{13}|=g^3_9=3g^1_3$ which shows that every complete $g^5_{13}$ is of the form $|K_X-3g^1_3|$ and is very ample. 
Therefore we have an isomorphism
\begin{align*}  &\Mm^1_{g,3}\setminus\Mm^1_{g,2}&\to\hskip18pt&\Gg_3\subset\Gg^5_{13} \\
&\hskip 18pt([X],g^1_3)&\mapsto\hskip18pt &([X], |K_X-3g^1_3|),
\end{align*} thus $\Gg_3:=\Ff_{3,0}\cup\Ff_{3,2}\cup\Ff_{3,4}\subset\Gg^5_{13}$ consists of all the very ample $g^5_{13}$'s on trigonal curves. $\Gg_3$ is irreducible by the irreducibility of $\Mm^1_{g,3}$, thus there is an irreducible family $\Hh_3\subset\HO{13,12,5}$ which is $\Aut(\PP^5)$-bundle over $\Gg_3$ surjecting onto $\Mm^1_{g,3}\setminus\Mm^1_{g,2}$, $\mu^{-1}\mu(X)=\Aut(\PP^5)X$ for $X\in\Hh_3$. 
\item[(d)] Let $X^\vee\subset\PP^3$ be the dual curve of $X\subset\PP^5$ -- which is by definition the curve induced by the moving part of  $|K_X(-1)|$. 
For $X\in\Ff_{4,0}$,  $|K_X(-1)|=|K_Q+X-\Oo_Q(1,2)|_{|X}=|\Oo_Q(1,1)|_{|X}$.   For $X\in\Ff_{4,2}$, $|K_X(-1)|=|\Oo_{\FF_2}(h+2f)|_{|X}$; cf. \eqref{f4}. Since $|\Oo_{\FF_2}(h+2f)|$ induces a birational morphism onto a quadric cone $Q_0\subset\PP^3$, $X^\vee$ is a smooth curve which is a residual to a line in a complete intersection of $Q_0$ with a quintic surface. Since a curve in $Q_0$ is a specialization of curves in smooth quadric $Q$ (\cite[Propositions 1.6 \& 2.1]{Zeuthen}), $\Ff_{4,2}$ is in the boundary of $\Ff_{4,0}$ and hence $\Gg_4:=\Ff_{4,0}\cup\Ff_{4,2}$ is an irreducible family with $\dim\Gg_4=\dim\Ff_{4,0}=23$. 

\item[(e)] Note that $\dim\Gg_3>\dim\Gg_4$. By semi-continuity of gonality, we deduce that 
 $\Gg_3$ and $\Gg_4$  are two distinct {\it components} of $\Gg^5_{13}$.
\item[(f)] For every  $X\in \Hh_{4,0}$ -- the component of $\HO{13,12,5}$ sitting over $\Ff_{4,0}$ -- $\mu (X)$ has a unique $g^1_4$, a unique $g^1_5$ and hence a unique (very ample) $g^5_{13}=|2g^1_4+g^1_5|$ by Proposition \ref{z5}.  Thus $\mu^{-1}\mu(X)=\Aut (\PP^5)X$ for every  $X\in\Hh_{4,0}$.
We also have $\mu^{-1}\mu(X)=\Aut(\PP^5)X$ for every  $X\in\Hh_{4,2}$ since $|\Oo_X(1)|=|\Oo_{\FF_2}(4h+9f)\tensor\Oo_X|$. 

\end{itemize}
\vnii Since $X\in\HO{13,12,5}$ is an extremal curve, $X$ is ACM with the Hilbert function
$h^0(\PP^n,\Ii_X(t)) =\binom{5+t}{5} -13t-1+12,  t\ge 2$; cf. Remark \ref{basic1} \eqref{exthil}.
\begin{itemize}
\item[(iii)] $g=11$:  Since $\pi(13,6)=9<g$, $\HO{13,g,5}=\HL{13,g,5}$. By \cite[Proposition 3.3]{lengthy}, $\HO{13,g,5}$ is irreducible of the expected dimension. A general $X\in\HO{13,11,5}$ has a plane model $C$, $\deg C=7$ with $4$ nodes induced by $|K_X(-1)|=g^2_7$ which corresponds  to a general element of the Severi variety $\Sigma_{7,11}$.
The normalization of a general element of ${\Sigma_{7,11}}$ has a unique \bpf $g^2_7$ \cite[Theorem 3.9]{treger}, hence the very ample $g^5_{13}=|K_X(-g^2_7)|$ is {\it unique}. Thus $\mu^{-1}\mu(X)=\Aut(\PP^5)X$ for  a general $X\in\HO{13,11,5}$ and $X$ is $5$-gonal by \cite{Coppens0}.

For the Hilbert function of a general $X\in \HO{13,11,5}$,  we first show that 
$h^1(\Ii_X(2))=0$. Assume $h^1(\Ii_X(2)) >0$, i.e. assume $h^0(\Ii_X(2))=h^0(\Oo_{\PP^5}(2))-h^0(X,\Oo_X(2))+h^1(\Ii_X(2))\ge 6$. By the same routine as we did in the proof of Theorem \ref{d=15} (iv-I), 
we may deduce that $X\subset T$, $\deg T=4$. 
However, this is not possible under our current situation $(d,g)=(13,11)$ by Remark \ref{cone1}-(0), (1b), (2c).
Thus, $h^1(\Ii_X(2))=0$ follows. By Remark \ref{pre2} (iii), $X$ is ACM. 
\item[(iv)]
If $g=10$, $\HL{13,10,5}$ is irreducible of the expected dimension by \cite[Theorem 2.3]{JPAA}.  Assume the existence of a component $\Hh\neq\HL{13,10,5}$.  For a general $X\in\Hh$, $\dim|\Oo_X(1)|\ge 6$ and hence $|K_X(-1)|=g^s_5, s\ge 2$.  If $|K_X(-1)|$ is biratioanally very ample, we have $g\le 6$, an absurdity.  If $|K_X(-1)|$ is compounded with  non-empty base locus, $X$ is hyperelliptic which does not carry a 
special very ample linear series $g^5_{13}$. Thus $\Hh_{13,10,5}=\HL{13,10,5}$ is irreducible. Since $|K_X(-1)|=g^1_5$,  we have 
$\mu(\HO{13,10,5})\subset\Mm^1_{g,5}$.
For a general $5$-gonal curve $[X]\in\Mm^1_{g,5}$, $|K_X(-g^1_5)|=g^5_{13}$ is very ample; otherwise  $|g^1_5(p+q)|=g^2_7$ for some $p+q\in X_2$, which has the same Clifford index as $g^1_5$, 
contradicting the fact that the Clifford index of a general $5$-gonal curve is only computed by the unique $g^1_5$; cf. \cite[Theorem]{B}. Conversely, on a general $X\in\HO{13,10,5}$, every very ample (\& complete) $g^5_{13}$ is of the form $|K_X-g^1_5|$. Thus we have $\Gg\stackrel{bir}{\cong}\Mm^1_{g,5}$,  $\HO{13,10,5}$ dominates $\Mm^1_{g,5}$ and $\mu^{-1}\mu(X)=\Aut(\PP^5)X$ for general $X\in\HO{13,10,5}$.

For the computation of the Hilbert function of a general $X\in \HO{13,10,5}$,  we claim that 
$h^1(\Ii_X(2))=0$. Assume $h^1(\Ii_X(2)) >0$. 
i.e. assume $h^0(\Ii_X(2))=h^0(\Oo_{\PP^5}(2))-h^0(X,\Oo_X(2))+h^1(\Ii_X(2))\ge 5$. Since $X$ is linearly normal, the restriction map $H^0(\Ii_X(2)) \to H^0(H,\Ii_{X\cap H,H}(2))$ is an isomorphism by Remark \ref{pre2}\,(iv), and hence $h^0(H,\Ii_{X\cap H,H}(2))\ge 5$.
By Remark \ref{aa1}, $X\cap H$ is contained in a rational normal curve or an elliptic normal curve $C\subset H$.  
By the same routine, we have $X\subset T$, a surface with $\deg T=4,5$. The case $\deg T=4$, 
 is not possible under our current situation $(d,g)=(13,10)$ by Remark \ref{cone1}-(0), (1b), (2c).
 Assume $X\subset T$, $\deg T=5$ and $T\cong\PP^2_4\subset\PP^5$ is smooth. Indeed, there is a very ample $L=(7; 3, 2^2,1)$ such that $X\in L$, $\deg X=L\cdot -K_T=13, g=10$. However, by \eqref{nofull}, the family of smooth
curves $X$ lying on a smooth del Pezzo does not constitute a full component. By Remark \ref{pre2} (iii), $X$ is ACM.


\vnii
\item[(v)] $g=9$:  By \cite[Theorem 2.2]{JPAA} and 
\cite[Theorem 2.1]{lengthy}
$\HL{13,g,5}$ is irreducible of the expected dimension. Suppose there exist very ample $\Dd'=g^6_{13}$; $|K_X-\Dd'|=g^1_3$. On a trigonal curve of genus $g\ge 7$, $|K_X-g^1_3|=g^6_{13}$ is very ample. Therefore
the family $\Ff\subset\Gg^5_{13}$ consisting of incomplete $g^5_{13}$'s sitting over $\Mm^1_{g,3}$ has dimension
$\dim\Ff=\dim\Mm^1_{g,3}+\dim\mathbb{G}(5,6)=2g+7<\lambda(13,g,5),$
which does not constitute a full component of $\HO{13,9,5}$. 
Thus $\Hh_{13,9,5}=\HL{13,9,5}$ is irreducible.

Let $X_K\subset\PP^8$ be the canonical curve of a non-hyperelliptic curve $X$ of genus $g=9$. For a general choice $p+q+r\in X_3$, the 
the projection $\pi_L: X_K\to \pi_L(X_K)\subset\PP^5$ with center at $L=\bar{pqr}$ is an embedding onto a linearly normal curve with $\deg\pi_L(X_K)=\deg K_X-3=13$. Conversely, any linearly normal smooth curve 
with $(d,g)=(13,9)$ in $\PP^5$ is obtained this way.
For a general $X\in \HL{13,9,5}$ -- which is necessarily non-hyperellitic and non-trigonal -- 
$G^5_{13}(X)\stackrel{\eta}{\cong} X_3$ where $\eta(\Dd)=|K_X-\Dd|\in X_3$.  Thus $\mu^{-1}\mu(X)=G^5_{13}(X)\times\Aut(\PP^5)X\cong X_3\times
\Aut(\PP^5)X$.


The restriction map $H^0(\PP^5, \Oo(t))\to H^0(X, \Oo_X(t))$ is of maximal rank by \cite{l8}; $\rho(13,g,5)>0$. 
Since $X$ is non-degenerate, $h^1(\Ii_X(1))=0$, $h^1(\Ii_X(t))=0$ for all $t\ge 2$ by Remark \ref{lars} \eqref{hil} and $X$ is ACM. 
\item[(vi)] If $g\le 8$, the irreducibility of $\HO{13,g,5}$ follows from Remark \ref{trivial}, which is a result in  \cite[Theorem 2.1]{PAMS}. In the proof (loc. cit.), it is shown that $\HO{13,g,5}$ dominates $\Mm_g$ and for a 
general $X\in\HO{d,g,r}$, $|\Oo_X(1)|$ is non-special.
Thus,  $\Gg\subset\Gg^5_{13}$ which $\HO{13,g,5}$ sits over is the Grassmannian bundle
with fibre $\mathbb{G}(5,13-g)$ over $\mathcal{J}_{d}=\mathcal{W}^{13-g}_d$, where 
$\mathcal{J}_{d}$ is the universal Jacobian of degree $d$. Thus, for a general
$X\in \HO{13,g,5}$, we have  
$\mu^{-1}\mu(X)=\Jac_{13}(X)\times\overset{o}{\mathbb{G}}(5,13-g)\times\Aut(\PP^5)X; \,2\le g\le 8.$

For Hilbert function of a general $X\in\HO{13,g,5} ~\mathrm{with}~ g= 8$, by the same reason 
as $g=9$ case (v), $X$ is ACM. 
For $g\le 7$, $h^1(\Ii_X(1))=8-g$,  $h^1(\Ii_X(t))=0$ for all $t\ge 2$ and $X$ is {\it not} ACM. 
\end{itemize}
\vspace*{-\baselineskip}
\end{proof}
\vspace{-12pt}
\vspace{-12pt}
\section {Curves of degree $d=12$ in $\PP^5$ }

\begin{proposition}\label{d=12}
\begin{itemize}
\item[\rm{(i)}] $\HO{12,g,5}=\emptyset$ if $g\ge 11$.
\item[\rm{(ii)}]  If $g=10$, $\HO{12,g,5}=\HL{12,g,5}\neq\emptyset$ is reducible with two components.
Every smooth  $X\in\HO{12,g,5}$ is ACM and $\mu^{-1}\mu(X)=\Aut(\PP^5)X$.
\item[\rm{(iii)}]  If $g=9$, $\HO{12,g,5}=\HL{12,g,5}\neq\emptyset$ is irreducible, a general $X\in\HO{12,g,5}$ is ACM, $\mu$ dominates ${\Mm^1_{g,4}}$ and $\mu^{-1}\mu(X)=\Aut(\PP^5)X$.
\item[\rm{(iv)}]  If $g=8$, $\HO{12,g,5}=\HL{12,g,5}$ is irreducible, a general  $X\in \HO{12,g,5}$ is ACM,
$\mu^{-1}\mu (X)=X_2\times\Aut(\PP^5)X$, $\mu(\HO{12,g,5})=\Mm_g\setminus\Mm^1_{g,2}$.
\item[\rm{(v)}]  $\HO{12,g,5}=\HL{12,g,5}$ is irreducible if $g=7$.
If $g\le 6$, $\HO{12,g,5}$ is irreducible and $\HL{12,g,5}=\emptyset$.
A general $\HO{12,g,5}$ is ACM (resp. not ACM) for $g=7$ (resp. for $2\le g\le 6$), 

$\mu^{-1}\mu(X)=\Jac_{12}(X)\times\overset{o}{\mathbb{G}}(5,12-g)\times\Aut(\PP^5)X; 2\le g\le 7.$
\end{itemize}
\end{proposition}
\begin{proof}
\begin{itemize}
\item[(i)] 
Since $g\le\pi(12,5)=10$, $\HO{12,g,5}=\emptyset$ if $g\ge 11$.
\item[(ii)] If $g=10$, $X\in\Hh_{12,10,5}$ is an extremal curve lying on a scroll, a rational normal cone or a Veronese.

\item[(a)]
In case $X$ lies on a Veronese surface, $X$ is the image of a smooth plane sextic. Let $\Hh_0$ be the irreducible family consisting of the images of smooth plane sextics under the Veronese embedding in $\PP^5$. Let $$\Ff_{0}:=\{ ~|\Oo_C(2)|~ |~  [C]\in\PP(H^0(\PP^2,\Oo(6)))/\Aut(\PP^2)\}\subset\Gg^5_{12}. $$
We have 
\begin{align*}
\dim\Hh_0&=\dim\Ff_0+\dim\Aut(\PP^5)\\
&=\dim\PP(H^0(\PP^2,\Oo(6)))-\dim\Aut(\PP^2)+\dim\Aut(\PP^5)\\&
=\Xx(12,10,5)=54.
\end{align*}
For a smooth curve $X\subset S\subset\PP^5$  of even degree $d=2a=12$
lying on a Veronese surface $S$,  by \eqref{plane}, we have
$h^1(N_{X,\PP^5})=0$
implying
$$h^0(N_{X,\PP^r})=\Xx(d,g,r)+h^1(N_{X,\PP^r})=\dim\Hh_0.$$
Thus $\Hh_0$ is a {\it generically reduced component} of the expected dimension. 
Since
a smooth plane curve of degree $a\ge 4$ has a unique $g^2_a$ and a unique $g^5_{2a}$, $\mu^{-1}\mu(X)=\Aut(\PP^5)X$ for any smooth $X\in\h{H}_0$.

\item[(b)]
Assume $X\subset S$, where $S$ is a smooth rational normal surface scroll; $S\cong\FF_0$ or $\FF_2$. By solving \eqref{sd} and \eqref{sg}, we have $X\in|3H|$; $\dim|3H|=27$ by \eqref{sdn}. 
If $X\subset\FF_0\cong \PP^1\times\PP^1=Q$, we have $X\in|\Oo_Q(3,6)|$ and $m(X)=4$ by Lemma \ref{maroni} and \eqref{F0}. 
If $X\subset\FF_2$, we have $X\in|3h+9f|$ and $m(X)=3$ 
by Lemma \ref{maroni} and \eqref{F2}. In both cases, $X$ is trigonal, $|L|_{|X}=g^1_3$, $|K_X|=|H+2L|_X$ and $|H|_{|X}=|K_X-2g^1_3|$. On any trigonal curve of genus $g\ge 10$, $|K_X-2g^1_3|$ is very ample, thus  every member in the family $$\Ff:=\{|K_X-2g^1_3| | X\in{\Mm^1_{g,3}\setminus}\Mm^1_{g,2}\}\subset\Gg^5_{12}$$ is very ample. 


\vni

\vni
\item[(c)]
Assume that $X\subset S_0$, a cone over a rational quartic in $H\cong\PP^4$.  By \eqref{conevertex1} we have $k=3, d=12=(r-1)k$, $m=0$, the ruling cut out $g^1_3$, $X$ does not pass through the vertex of $S_0$,  $\tilde{X}\cong X$, $3C_0+12f\equiv\tilde{X}$, where $\tilde{X}$ is the strict transformation of $X$ under $\FF_4\cong\tilde{S}_0\stackrel{{|C_0+4f|}}{\to}S_0\subset\PP^5$.
Since  $|K_{\FF_4}+\tilde{X}-\phi^*{(\Oo(1))}|=|2f|$, we have $|\Oo_X(1)|=|K_X-2g^1_3|$. By setting
$$\Ff_{cone}:=\{|\Oo_X(1)|~| X\subset S_0\subset\PP^5 \}\subset\Gg^5_{12},$$ 
we have $\Ff_{cone}\subset\Ff\subset{\Gg^5_{12}}$. Thus, curves $X\subset\PP^5$ with $(d,g)=(12,10)$ lying on a quartic surface scroll or a cone is trigonal 
which forms an irreducible family $\Hh\subset\HO{12,10,5}$ sitting over $\Ff$.
 By semi-continuity of gonality, $\Hh$ is a {\it component} $\&$
$\Hh\neq\Hh_0$.
For a smooth $X\in\Hh$ we have
$\mu^{-1}\mu(X)=\Aut(\PP^5)X$. By Remark \ref{basic1}, $X$ is ACM.




\item[(iii)]
If $g=9$, the irreducibility of $\HL{12,9,5}$ is a special case of \cite[Theorem 2.3]{JPAA} and
\cite[Theorem 2.1]{lengthy}. Suppose there is a component $\Hh\neq\HL{12,9,5}$. For a general $X\in\Hh$,
we have $|K_X(-1)|=g_{4}^\beta$ with $\beta\ge 2$, which is impossible by Clifford theorem since $X$ cannot be hyperelliptic. Thus $\HO{12,9,5}=\HL{12,9,5}$ is irreducible of the expected dimension. For a general $X\in\HL{12,9,5}$, $|K_X(-1)|=g^1_4$ is base-point-free, for otherwise 
$|K_X-g^1_4|=|K_X-(g^1_3+p)|$ is not very ample. Conversely, on a general $4$-gonal curve $X$ of genus $g=9$ (with a unique $g^1_4$), $|g^1_4+p+q|=g^1_6$ for any $p+q\in X_2$ by \cite[Theorem]{B}. Thus $|K_X-g^1_4|$ is the {\it unique} very ample $g^5_{12}$ and $\HO{12,9,5}$ dominates $\Mm^1_{g,4}$.
Thus, $\mu^{-1}\mu(X)=\Aut (\PP^5)X$ for general $X\in\HO{12,9,5}$.
\vnii
Claim; $h^1(\Ii_X(2))=0$ for a general $X\in\HO{12,9,5}$. 
Assume $h^1(\Ii_X(2)) >0$. By an argument parallel to the one we used in the proof of Theorem \ref{d=15}[(iv)-(I)], $X\subset T$ with  $\deg T=4$.


\begin{enumerate}
\item[(a)] If $T\cong\FF_0\cong Q$, $X\in|\Oo_Q(4,4)|$, $X$ has two distinct $g^1_4$'s, contradicting $X$ being a general $4$-gonal curve. 
\item[(b)] Assume
that $T$ is a cone over a rational normal curve. By Remark \ref{cone1}(1b), $X$ is trigonal, 
a contradiction by Castelnuovo-Severi inequality. 
\item[(c)]
Assume $T\cong \FF_2$ and set $X\in |\Oo_{\FF_2}(ah+bf)|$. We have 
$(a,b)=(4,8)$ from \eqref{F2}.
Since $\dim|4h+8f|-\dim\Aut(\FF_2)+\dim\Aut(\PP^5)=52<\Xx(12,9,5)=56,$ the family arising this way
does not constitute a full component. 
\end{enumerate}
By the Claim together with Remark \ref{pre2} (iii), $X$ is ACM.

\item[(iv)] We may argue in the following way as we did in Theorem \ref{d13} (v).
Let $X_K\subset\PP^7$ be the canonical curve of a non-hyperelliptic curve $X$ of genus $g=8$. For any choice $p+q\in X_2$, the 
the projection $\pi_L: X_K\to \pi_L(X_K)\subset\PP^5$ with center $L=\bar{pq}$ is an embedding, $\deg\pi_L(X_K)=\deg K_X-2=12$. Conversely, any smooth curve 
with $(d,g)=(12,8)$ in $\PP^5$ is obtained this way.
Thus for a general $X\in \HL{12,8,5}$ -- which is necessarily non-hyperellitic -- 
$G^5_{12}(X)\stackrel{\eta}{\cong} X_2$ where $\eta(\Dd)=|K_X-\Dd|\in X_2$.  Thus $\mu^{-1}\mu(X)=G^5_{12}(X)\times\Aut(\PP^5)\cong X_2\times
\Aut(\PP^5)X$.
Therefore it follows that $\HL{12,8,7}= \HO{12,8,7}$ is irreducible, $\mu$ surjects onto $\Mm_g\setminus\Mm^1_{g,2}$ and  $\dim\HO{12,8,5}=\dim\Mm_g+2+\dim\Aut(\PP^5)=\Xx(12,8,5).$

\vnii

Take a general $X\in\HO{12,g,5} ~\mathrm{with}~ g\le 8$.  Since $\rho(12,g,5)>0$, 
the restriction map $H^0(\PP^5, \Oo(t))\to H^0(X, \Oo_X(t))$ is of maximal rank by \cite{l8} with the Hilbert function is given by \eqref{hil} for all $t\ge 2$; cf. Remark \ref{lars}. $X$ is clearly ACM.
\item[(v)] We may copy the proof of Theorem \ref{d13} (vi). 
If $g\le 7$, the irreducibility of $\HO{12,g,5}$ follows from Remark \ref{trivial}. In this case,  $\Gg\subset\Gg^5_{12}$ over which $\HO{12,g,5}$ sits is just the Grassmannian bundle
with fibre $\mathbb{G}(5,12-g)$ over $\mathcal{J}_{d}=\mathcal{W}^{12-g}_d$. Thus, for a general
$X\in \HO{12,g,5}$, we have  
$$\mu^{-1}\mu(X)=\Jac_{12}(X)\times\overset{o}{\mathbb{G}}(5,12-g)\times\Aut(\PP^5)X; \,2\le g\le 7.$$
A general $X\in\HO{12,g,5}$ is ACM for $g=7$.  For a general $X\in\HO{12,g,5}, \,g\le 6$, we have $h^1(\Ii_X(t))=0$ $\forall t\ge2$,  and $h^1(\Ii_X(1))=7-g$.
\end{itemize}
\vspace*{-\baselineskip}
\end{proof}
\vspace{-12pt}
\section{Curves of  degree $d\le11$ in $\PP^5$}
\subsection {Curves of degree $d=11$ in $\PP^5$ } 
\begin{proposition} 
\begin{itemize}
\item[\rm{(i)}] 
$\HO{11,g,5}=\emptyset$ if $g\ge9$.
\item[\rm{(ii)}]  $\HO{11,8,5}=\HL{11,8,5}$ is irreducible of the expected dimension. Every smooth $X\in\HO{11,8,5}$
is trigonal, ACM and
$\mu^{-1}\mu(X)=\Aut (\PP^5)X.$
\item[\rm{(iii)}]  $\HO{11,7,5}$ is irreducible of the expected dimension, $\mu(\HO{11,7,5})=\Mm_{7}\setminus\Mm^1_{7,3}$, every $X\in\HO{11,7,5}$ is ACM and $\mu^{-1}\mu(X)\cong X\times\Aut(\PP^5)X$.
\item[\rm{(iv)}]  If $g\le 6$, $\HO{11,g,5}$ is irreducible of the expected dimension $\mu(\HO{11,g,5})=\Mm_g$. For a general $X\in\HO{11,g,5},$ 
$$\mu^{-1}\mu(X)=\Jac_{11}(X)\times\overset{o}{\mathbb{G}}(5,11-g)\times\Aut(\PP^5)X; 2\le g\le 6.$$
$X$ is ACM only if $g=6$.

\end{itemize}
\end{proposition}
\begin{proof} 

\begin{itemize}
\item[(i)] Since $\pi(11,5)=8$, $\HO{11,g,5}=\emptyset$ if $g\ge 9$.
\item[(ii)] If $g=8$, $X\in\Hh_{11,8,5}$ is an (linearly normal) extremal curve, $\HO{11,8,5}=\HL{11,8,5}$.
Every extremal curve is ACM,  $h^1(\Ii_X(t))=0$ for all $t\ge 1$. For every smooth $X\in\Hh_{11,8,5}$, $|K_X(-1)|=g^1_3$ and is trigonal. Conversely, every trigonal curve $X$ with $g=8$, $|K_X-g^1_3|=g^5_{11}$ is very ample and vice versa, i.e. every very ample $g^5_{11}$ is of the form $|K_X-g^1_3|$. 
Thus, $\HO{11,8,5}$ dominates  $\Mm^1_{g,3}$ and is irreducible by the irreducibility of $\Mm^1_{g,3}$. 
By the {\it uniqueness} of a very ample $g^5_{11}$, 
$\mu^{-1}\mu(X)=\Aut (\PP^5)X$.


\item[(iii)] If $g=7$, by  \cite[Theorem 2.2]{JPAA}, $\HL{g+r-1,g,r}$ is irreducible having the expected dimension if $g\ge r+1$. Suppose there is a component $\Hh\subset\HO{11,7,5}$ such that $\Hh\neq\HL{11,7,5}$.  The hyperplane series $\Dd\subset|\Oo_X(1)|$ of a general $X\in\Hh$ is an incomplete $g^5_{11}$, thus $|\Dd|= g^\gamma_{11}$, $\gamma\ge6$. By Clifford's theorem $|\Dd|$ is non-special and hence
$\dim|\Dd|=11-g=4=\gamma,$ a contradiction. Thus $\HO{11,7,5}=\HL{11,7,5}$ and every smooth $X\in\HL{11,7,5}$ is a  projection from a point on the canonical curve of $X$. Recall that the canonical curve of a trigonal curve has one dimensional family of triscant lines so that the projection from a point on the canonical curve is singular in $\PP^5$. Thus $\mu(\HL{11,7,5})=\Mm_g\setminus\Mm^1_{g,3}$ and $\mu^{-1}\mu(X)\cong X\times\Aut(\PP^5)X$
for any $X\in\HO{11,7,5}$.

Obviously, $h^i(\Ii_X(1))=0, i=0,1$. Since $\rho(11,7,5)>0$, 
the restriction map $H^0(\PP^5, \Oo(t))\to H^0(X, \Oo_X(t))$ is of maximal rank by \cite{l8}
with the Hilbert function is given by \eqref{hil} for all $t\ge 2$; cf. Remark \ref{lars}. 
\item[(iv)]
If $g\le 6$, the irreducibility and the dominance of the moduli map $\mu$ follow from Remark \ref{trivial}. Since $\pi(11,6)=5$,
$\HO{11,6,5}=\HL{11,6,5}\neq\emptyset$ if $g=6$ whereas $ \HL{11,g,5}=\emptyset$
and $\HO{11,g,5}\neq\emptyset$ if $g\le 5$. For a general $X\in\HO{11,g,5}$, by the same way as the proof of Theorem \ref{d13}\,(vi),  we have
$$\mu^{-1}\mu(X)=\Jac_{11}(X)\times\overset{o}{\mathbb{G}}(5,11-g)\times\Aut(\PP^5)X; \,2\le g\le 6.$$

\vni Obviously, $h^0(\Ii_X(1))=0$. Since $\rho(11,g,5)>0$ for $g\le 6$, 
the restriction map $H^0(\PP^5, \Oo(t))\to H^0(X, \Oo_X(t))$ is of maximal rank by \cite{l8}
and the Hilbert function is given by \eqref{hil} for all $t\ge 2$; cf. Remark \ref{lars}. 
$X\in \HO{11,g,5}$ is ACM (resp. not ACM) if $g=6$ (resp. if $g\le 5$). 
\end{itemize}
\vspace*{-\baselineskip}
\end{proof}


\subsection{Curves of degree $d\le10$ in $\PP^5$.}

\vni
The following is a descriptive version of elementary part of our discussion so far which we state as a proposition.
\begin{proposition}
Every non-empty $\HO{d,g,5}$ is irreducible of the expected dimension $\Xx(d,g,5)$ and dominates $\Mm_g$ if $d\le 10$. More precisely:
\begin{itemize}
 
\item[\rm{(i)}]  $\HO{10,g,5}\neq \emptyset$ if and only if $g\le 6$. $X\in\HO{10,g,5}$ is
ACM (resp. not ACM) if $g=5, 6$ (resp. if $g\le 4$). 
\[
\mu^{-1}\mu(X)=
\begin{cases}\Aut(\PP^5)X \mathrm{ ~for ~every ~smooth~ } X\in\HO{10,g,5};\, g=6\\
\Jac_{10}(X)\times\overset{o}{\mathbb{G}}(5,10-g)\times\Aut(\PP^5)X; \,2\le g\le 5 \\\mathrm{for~ general~} X\in\HO{10,g,5}, 
\end{cases}
\] 
with the Hilbert function given by \eqref{hil}.
\item[\rm{(ii)}]  For $5\le d\le 9$, $\HO{d,g,5}\neq \emptyset$ if and only if $0\le g\le d-5$.
A general $X\in\HO{d,g,5}$ is
ACM if and only  if $d=g+5$. For a general $X\in\HO{d,g,5}$, 
$$
\mu^{-1}\mu(X)=
\Jac_{d}(X)\times\overset{o}{\mathbb{G}}(5,d-g)\times\Aut(\PP^5)X; \,2\le g\le d-5,
$$
with the Hilbert function  given by \eqref{hil}. 

\end{itemize}
\end{proposition}
\begin{proof}
\begin{itemize}
\item[(i)]For $d=10$, we have $g\le\pi(d,5)=6$ and $\rho(d,g,5)\ge 0$. By \cite[Theorem, pp. 69--70]{he}, there is a unique component of $\HO{d,g,5}$ dominating $\Mm_g$. Since every $X\in\HO{10,6,5}=\HL{10,6,5}$ is a canonical curve of a non-hyperelliptic curve, $\HO{10,6,5}$ is irreducible and $\mu(\HO{10,6,5})=\Mm_6\setminus\Mm_{6,2}^1$. For $g\le 5$, $\HO{10,g,5}$ irreducible by 
Remark \ref{trivial}; $\HO{10,5,5}=\HL{10,5,5}$ and $\HO{10,g,5}\neq\emptyset ~~\&~~\HL{10,g,5}=\emptyset$ if $g\le 4$. The description of $\mu^{-1}\mu(X)$ is obvious and the Hilbert function is given by \eqref{hil} since $\rho(10,g,5)\ge 0$; Remark \ref{lars}. 
\item[(ii)]  For $5\le d\le 9$, $\pi (d,5)=d-5$ and $|\Oo_X(1)|$ is non-special. Thus $\HO{d,g,5}\neq\emptyset$ if and only if  $0\le g\le d-5$ and is irreducible by Remark \ref{trivial}. \qed
\end{itemize}
\vspace*{-\baselineskip}
\end{proof}
We have the following easy description of the fibre of moduli map of the Hilbert scheme of non-degenerate {\it rational or elliptic curves} of {\it any degree} $d$.
\begin{remark}
\noindent
(i) $g=1$; $\mu^{-1}\mu(X)=\Aut(\PP^5)X\times\overset{o}{\mathbb{G}}(5,d-1).$

\noindent
(ii) $g=0, d\ge 5$; $\mu^{-1}\mu(X)=\Aut(\PP^5)/\Aut(\PP^1)\times\mathbb{G}(5,d)$.
\end{remark}

\begin{remark}
(i) We have seen that for $d\le 13$, a general $X\in\HL{d,g,5}$ is ACM  
as long as $\HL{d,g,5}\neq\emptyset$. This is partly due to the fact that the range of the genera $g\le \pi(d,5)$ such that $\rho(d,g,5)\ge0$ rather big and being ACM follows for free by the maximal rank property of the curve 
in question. 

\vnii
(ii) For $d=14$, a general $X\in\HL{14,g,5}$ is ACM except $g=14$. We remark that a smooth 
curve in $\PP^5$ with $(d,g)=(14,14)$ is nearly extremal, which forces $X$ a non ACM curve by 
\cite[Corollary 3.17]{he} as we have mentioned earlier.

\vnii
(iii) Incidentally, for $d=15$, there is no nearly extremal curve. However, there are components whose 
general member is not ACM, e.g. $\Gamma_i\subset \HO{15,16,5}$, $i=1,2$; cf. Theorem \ref{d=15}\,(iii).

\vnii
(iv) For larger values of $\deg X\ge 16$, we anticipate that there would appear more Hilbert scheme $\HO{d,g,5}$ -- with relatively high genus -- having  non-ACM curves as its general member.  
\end{remark}


\vspace{-18pt}
\begin{thebibliography}{99}
\bibitem{Accola1}
{R. Accola},
{\it Topics in the theory of Riemann surfaces.}
Lecture Notes in Mathematics 1595, Springer, Heidelberg, 1991.

\bibitem{Antonelli}{V. Antonelli}, \textit{Characterization of Ulrich bundles on Hirzebruch surfaces.} Revista Matematica Complutense \textbf{34} (2021), 43--74.
\bibitem{AC1}{E. Arbarello and M. Cornalba},\textit{Footnotes to a paper of Beniamino Segre.} Math. Ann. \textbf{256} (1981), 341--362.
\bibitem{AC2}
{E. Arbarello and M. Cornalba},
\textit{A few remarks about the variety of irreducible plane curves of given degree and genus.} Ann. Sci. \'Ec. Norm. Sup\'er. (4) \textbf{16} (1983), 467--483.
\bibitem{ACGH}%
{E. Arbarello, M. Cornalba, P. Griffiths and J. Harris},
\textit{Geometry of Algebraic Curves Vol.I.}
Springer-Verlag, Berlin/Heidelberg/New York/Tokyo, 1985.
\bibitem{ACGH2}%
{E. Arbarello, M. Cornalba and  P. Griffiths},
\textit{Geometry of Algebraic Curves Vol.II.}
Springer, Heidelberg, 2011.
\bibitem{B}
{E. Ballico},
\textit{On the Clifford index of algebraic curves.} Proc. Amer. Math. Soc., \textbf{97} (1986), 217--218.

\bibitem{edinburgh}{E. Ballico and C. Keem}, \textit{On the Hilbert scheme of smooth curves of degree $15$ in $\mathbb{P}^5$}, To appear in \textsc{The Proceeding of the Royal Society of Edinburgh Section A; Mathematics}, 
available at http://arxiv.org/abs/2310.00682.
\bibitem{bumi24} {E. Ballico and C. Keem}, \textit{On the Hilbert scheme of smooth curves of degree $15$ and genus $14$ in $\mathbb{P}^5$}, \textsc{Bollettino UMI} \textbf{17} no.3 (2024), 5537--557, available at https://doi.org/10.1007/s40574-023-00396-2.
\bibitem{JPAA}%
{E. Ballico, C. Fontanari and C. Keem}, 
\textit{On the Hilbert scheme of linearly normal curves in $\mathbb{P}^r$ of relatively high degree.}
J. Pure Appl. Algebra \textbf{224} (2020), 1115--1123.
\bibitem{cil}
{C. Ciliberto},
\textit{On the Hilbert scheme of curves of maximal genus in a projective space.} 
Math. Z. \textbf{194} (1987), 351-- 363.
\bibitem{Coppens0}
{M. Coppens},
\textit{The gonality of general smooth curves with a prescribed plane nodal model.} 
Math. Ann. \textbf{289} (1991), 89--93.
\bibitem{Coppens}
{M. Coppens},
\textit{Embeddings of general blowing-ups at points,} 
J. reine angew. Math. \textbf{469} (1995), 179--198.



\bibitem{sandra}{S. Di Rocco},\textit{$k$-very ample line bundles on Del-Pezzo surfaces}.Math. Nachr. , \textbf{179} (1996), 47--56.



\bibitem{EH}{D. Eisenbud and J. Harris}
\textit{The practice of algebraic curves}, GSM 250 (2024), American Mathematical Society.
\bibitem{gieseker} D. Gieseker, \textit{Stable curves and special divisors: Petri's conjecture}, Invent. Math. \textbf{66} no. 2 (1982) 251?275.

\bibitem{he} J.~Harris, \textit{Curves in projective space}, S\'eminaire de Math\'ematiques Sup\'erieures, vol.~85, Presses de
l'Universit\'e de Montr\'eal, Montreal, Que., 1982, With the collaboration of D. Eisenbud. 
\bibitem{Zeuthen} {R. Hartshorne}, \textit{Families of curves in $\PP^3$ and Zeuthen's problem.} Memoirs of the Americal Mathematical Society, \textbf{130}, no. 617 (1997).
\bibitem{hh} R. Hartshorne and A. Hirschowitz, Smoothing algebraic space curves, Algebraic Geometry, Sitges
1983, 98--131, Lect. Notes in Math. 1124, Springer, Berlin 1985.

\bibitem{PAMS}
{C. Keem},
\textit{Reducible Hilbert scheme of smooth curves with positive Brill-Noether number} Proc. A.M.S., \textbf{122} (1994), 349--354.
\bibitem{lengthy} {C. Keem}, \textit{Existence and the reducibility of the Hilbert scheme of linearly normal curves in $\mathbb{P}^r$ of relatively high degrees,} J. Pure Appl. Algebra, \textbf{227} (2023), 1115--1123 Preprint, https://arxiv.org/abs/2101.00559.
\bibitem{speciality} {C. Keem}, \textit{On the Hilbert scheme of  linearly normal curves in $\mathbb{P}^r$ with small index of speciality.} Indagationes Mathematicae, \textbf{33} (2022), 1102--1124 https://arxiv.org/abs/2103.16321.
\bibitem{KKy2}{C. Keem and Y.-H. Kim},\textit{Irreducibility of the Hilbert Scheme of smooth curves in $\PP^4$ of degree $g+2$ and genus $g$.} Arch. Math., \textbf{109} (2017), 521--527.

\bibitem{KK3}{C. Keem and Y.-H. Kim},\textit{On the Hilbert scheme of linearly normal curves in $\mathbb{P}^4$ of degree $d = g+1$ and genus $g$.} Arch. Math., \textbf{113} (2019), no. 4, 373--384.
\bibitem{landscape} {C. Keem}, \textit{Hilbert scheme of smooth curves of degree sixteen in $\mathbb{P}^5$}, to appear in a special issue on Bollettino UMI, A Projective Geometry Landscape, 
http://arxiv.org/abs/2503.18428.
\bibitem{l8} E. Larson, The maximal rank conjecture, arxiv:1711.04906.
\bibitem{laz} R. Lazarsfeld, 
A sharp Castelnuovo bound for smooth surfaces. Duke Math. J. 55 (1987), no.2, 423--429.







\bibitem{juan} J. Migliore, \textit{Introduction to liaison theory and deficiency moudles}, Progress in Mathematics, vol.~165, Springer Science, 1998. 


\bibitem{treger} R. Treger, Plane curves with nodes, {\it {Canadian J. Math.}} {\bf {41}} (1989), no. 2, 193--212.
\end{thebibliography}
\end{document}